\documentclass[12pt,a4paper]{article}

\usepackage{graphicx}%
\usepackage{multirow}%
\usepackage{amsmath,amssymb,amsfonts}%
\usepackage{amsthm}%
\usepackage{mathrsfs}%
\usepackage[title]{appendix}%
\usepackage{xcolor}%
\usepackage{textcomp}%
\usepackage{manyfoot}%
\usepackage{booktabs}%
\usepackage{algorithm}%
\usepackage{algorithmicx}%
\usepackage{algpseudocode}%
\usepackage{listings}%
\usepackage{tikz}
\usepackage[font={small,it}]{caption}
\usepackage{dsfont}
\usepackage{hyperref}

\DeclareMathOperator*{\argmin}{arg\,min}

\newcommand{\e}{\varepsilon}
\newcommand{\N}{\mathbb{N}}
\newcommand{\R}{\mathbb{R}}
\renewcommand{\P}{\mathcal{P}}

\renewcommand{\d}{{\rm d}}
\newcommand{\g}{\boldsymbol{\gamma}}
\newcommand{\K}{\mathcal{K}}
\newcommand{\weakto}{\rightharpoonup}
\newcommand{\X}{\mathcal{P}_2(H)}

\theoremstyle{plain}
\newtheorem{theorem}{Theorem}[section]
\newtheorem{corollary}[theorem]{Corollary}
\newtheorem{proposition}[theorem]{Proposition}
\newtheorem{lemma}[theorem]{Lemma}
\theoremstyle{definition}
\newtheorem{definition}[theorem]{Definition}
\newtheorem{remark}[theorem]{Remark}

\title{Generalized Wasserstein Barycenters}
\author{Francesco Tornabene, Marco Veneroni, Giuseppe Savar\'e}

\begin{document}
\maketitle

\abstract{We study the existence and uniqueness of the barycenter of a signed distribution of probability measures on a Hilbert space. The barycenter is found, as usual, as a minimum of a functional. In the case where the positive part of the signed measure is atomic, we can show also uniqueness. In the one-dimensional case, we characterize the quantile function of the unique minimum as the orthogonal projection of the $L^2$-barycenter of the quantiles on the cone of nonincreasing functions in $L^2(0,1)$. Further, we provide a stability estimate in dimension one and a counterexample to uniqueness in $\R^2$. Finally, we address the consistency of the barycenters and we prove that barycenters of a sequence of approximating measures converge (up to subsequences) to a barycenter of the limit measure.}




\section{Introduction and main results}
In this paper, we introduce an extension of the barycenter problem studied by Agueh and Carlier \cite{agueh-carlier} 
in the Wasserstein space. Let $H$ be a separable Hilbert space, let $\X$ be the space of Borel probability measures on $H$ with finite second moment, endowed with the 2-Wasserstein distance, denoted $W_2$. Let  $\lambda$ be a signed measure on the Borel $\sigma$-algebra of $\X$, such that $\lambda(\X)=1$ (see, e.g., \cite{bogachev} for a comprehensive treatise of measure theory).  We study the minimization of the functional 
\begin{equation}
\label{eq:functionalE}
	\mathcal E(\mu) = \int_{\X} W_2^2(\mu,\nu)\,\d\lambda(\nu),
\end{equation}
which we call the \emph{generalized} population barycenter problem of the signed measure $\lambda$. When $\lambda$ is concentrated on a discrete set $\{\nu_1,\ldots,\nu_n\}\subset \X$ and it is nonnegative, i.e.
\[
	\lambda_i \geq 0,\text{ for } i=1,\ldots,n,\qquad \sum_{i=1}^n \lambda_i=1,\qquad \lambda= \sum_{i=1}^n \lambda_i \delta_{\nu_i},
\] 
we recover the usual barycenter functional introduced in \cite{agueh-carlier}:
\[
	\mathcal E(\mu) = \sum_{i=1}^n \lambda_i W_2^2(\mu,\nu_i).
\]
The extension of the barycenter, from $n$ discrete measures to a continuous distribution of probability measures (i.e., a probability distribution, or `population`, of probability distributions) was studied in \cite{bigot2012, bigot2018, legouic}. In our contribution, we allow for a signed measure, i.e., for negative weights, as long as $\lambda(\X)>0$ (or, without loss of generality, equal to one). This condition is necessary since, even in the discrete case, if $x_i\in H$ and $\sum_{i=1}^n\lambda_i <0$, then
\[
		\inf_{x \in H} \sum_{i=1}^n \lambda_i |x-x_i|^2=-\infty.
\]
We denote by $\lambda = \lambda^+ - \lambda^-$ the (unique) Hahn-Jordan decomposition of $\lambda$,  by $|\lambda|=\lambda^+ +\lambda^-$ the total variation measure of $\lambda$, and  by $m_2^2(\nu)$ the second moment of a measure $\nu\in \X$:
\[
	m_2^2(\nu):=\int_H |x|^2\d\nu(x) <\infty.
\]
We prove the existence of a minimizer of \eqref{eq:functionalE} and, in the particular case where $\lambda^+$ is concentrated on a single measure in $\X$, we are able to show uniqueness of the solution. 
\begin{theorem}
\label{th:newmain}
Let $H$ be a real separable Hilbert space and let $\lambda$ be a signed measure on $\X$ with 
\begin{equation}
\label{eq:main_hyp}
	\lambda(\X)=1\qquad \text{and}\qquad  \int_{\X}  m_2^2(\nu)\d|\lambda|(\nu)<\infty.
\end{equation}
Define $\mathscr{E}:\X\to \R$ by
\[
	\mathscr{E}(\mu)=\int_{\X}  W^2_2(\mu,\nu)\d\lambda(\nu).
\]
Then:
(i) There exists a solution $\bar \mu \in \X$ of the problem
\[	
	\min_{\mu \in \X} \mathscr{E}(\mu).
\]
(ii) If, in particular, supp$(\lambda^+)$ is an atom, then there exists a unique minimizer $\bar \mu\in \X$ of $\mathscr{E}$.
\end{theorem}
It should be noticed that uniqueness is false, in general, in the classical barycenter problem with positive weights in $H=\R^d$. Moreover, the bound on the quadratic moments of $|\lambda|$ in \eqref{eq:main_hyp} is the minimal assumption that ensures that $\mathcal{E}$ is finite on $\X$, but finite-energy sequences are not relatively compact in $\X$ if $H$ is infinite-dimensional. In order to recover the compactness of a minimizing sequence, which we need for the direct method of the calculus of variations, we make use of the weak topology on $\X$, introduced in \cite{naldi-savare}, which we denote $\P_2^w(H)$ - see Section \ref{ssec:weak}. The same framework is also crucial in the proof of the lower-semicontinuity of $\mathcal{E}$.

Let $(\lambda_k)_{k\in\N}$ be a sequence of signed measures on $\X$, such that 
\begin{equation}
\label{eq:hypotheses}
	\lambda_k(\X)=1,\qquad \sup_{k \in \N}M_2^2(|\lambda_k|)=\sup_{k \in \N} \int_{\X}m_2^2(\nu)\, \d|\lambda_k|(\nu)<+\infty.
\end{equation}
Since Theorem \ref{th:newmain} asserts that each $\lambda_k$ admits a generalized Wasserstein barycenter, it is natural to ask whether a sequence of barycenters of $\lambda_k$ converges to a barycenter of $\lambda$, if $\lambda_k \to \lambda$, in a suitable sense. We can provide an answer to this question in the framework of $\Gamma$-convergence (see, e.g., \cite{dalmaso}) on $P_2^w(H)$.
\begin{theorem}
\label{th:gamma}
Let $(\lambda_k)_{k \in \N}$ and $\lambda$ be signed measures on $\X$, satisfying \eqref{eq:hypotheses}. Assume that
\begin{align}
	\lim_{k\to \infty} \int_{\X}\varphi(\nu)\,\d\lambda^{\pm}_k(\nu) &= \int_{\X}\varphi(\nu)\,\d\lambda^{\pm}(\nu) \qquad \forall \varphi \in C^0_b(X),\label{hyp:wconv}\\
	\lim_{k\to \infty} \int_{\X}m_2^2(\nu)\,\d |\lambda_k|(\nu) &= \int_{\X}m_2^2(\nu)\,\d |\lambda|(\nu).\label{hyp:moment}
\end{align}
Let $\mathcal{E}_k:\X \to \R$ be defined as
\[
	\mathcal{E}_k(\mu):= \int_{\X} W_2^2(\nu,\mu)\d\lambda_k(\nu).
\]
Then

(i) $\mathcal{E}_k$ $\Gamma$-converges to $\mathcal{E}$ in $\X$, with respect to the convergence induced by the 2-Wasserstein distance, and also with respect to the weak topology of $\P_2^w(H)$;

(ii) If $\mu_k$  is a minimizer of $\mathcal{E}_k$ for all $k\in \N$, the sequence $(\mu_k)_{k\in \N}$ is precompact with respect to the weak topology of $\P_2^w(H)$;

(iii) As a consequence of (i) and (ii),  if $\overline{\mu} \in \X$ is a limit point in $\P_2^w(H)$ of a subsequence of minimizers $(\mu_{k_j})$, then
	\[
		\lim_{j \to \infty} \mathcal{E}_{k_j}(\mu_{k_j}) = \mathcal{E}(\overline{\mu})=\min_{\mu \in \X}\mathcal{E}(\mu).
	\] 
\end{theorem}
We remark that $\Gamma$-convergence with respect to the weak and the strong topologies is also known as Mosco-convergence (see \cite{mosco}). 

In the case $H = \R$, we provide an explicit characterization of the minimizer. Indeed, in the 1-dimensional case any probability measure $\mu\in \P_2(\R)$ can be represented by its quantile function (see, e.g., \cite[Theorem 2.18]{villani_TOT}), i.e., the pseudo-inverse $X_\mu$ of its distribution function $F_\mu$:
\begin{align*}
	F_\mu(x)&:= \mu((-\infty,x]), & &\forall\,x\in \R,\\
	X_\mu(w)&:=\inf \{ x: F_\mu(x)>w\}, & &\forall\,w\in (0,1).
\end{align*}
The map $\mu \mapsto X_\mu$ is an isometry between $\P_2(\R)$, endowed with the 2-Wasserstein distance, and the convex cone $\K$ of nondecreasing functions in the Hilbert space $L^2(0,1)$. We denote by $P_\K$ the orthogonal projection operator, $P_\K:L^2(0,1)\to \K$.
Our main result, in the 1-dimensional case, is the following.
\begin{theorem}
\label{th:intro}
Let $\lambda$ be a signed measure on $\P_2(\R)$, satisfying
\begin{equation}
\label{eq:main_hyp_1d}
	\lambda(\P_2(\R))=1\qquad \text{and}\qquad  \int_{\P_2(\R)}  m_2^2(\nu)\d|\lambda|(\nu)<\infty.
\end{equation}
Then there exists a unique solution $\bar \mu \in \P_2(\R)$ of
\[	
	\inf_{\mu \in \P_2(\R)} \int_{\P_2(\R)} W^2_2(\mu,\nu)\d\lambda(\nu),
\]
and it is characterized by
\begin{equation}
\label{eq:charact}
	X_{\bar \mu} = P_\K\left( \int_{\P_2(\R)} X_{\nu} \d\lambda(\nu)\right),
\end{equation}
where $X_{\nu}\in \K$ is the pseudo-inverse function of (the distribution function of) $\nu$.
\end{theorem}

\subsection{Literature review and related problems}
\subsubsection*{Wasserstein barycenters} The barycenter of two probability measures, with positive weights adding to 1, is well-known as McCann's interpolation \cite{mc_cann}, while a complete study of the Wasserstein barycenters of $n$ measures in $\R^d$ was done in \cite{agueh-carlier}, where, in particular, it was shown the existence of a minimizer and the uniqueness, in the case where one of the $\nu_i$'s vanishes on small sets. In the present paper we do not require the absolute continuity of one of the measures, but in order to obtain uniqueness we need the positive part of $\lambda$ to be concentrated on a single measure. In the one-dimensional case, if $\lambda$ is concentrated on a discrete set,  if all weights $\lambda_i$ are positive, and all measures $\nu_i$ are nonatomic, our characterization \eqref{eq:charact} reduces to the one-dimensional formula in \cite[Section 6.1]{agueh-carlier}.

The barycenter of a continuous distribution of probability measures on $\R^d$ (also called a \emph{population} barycenter) was first studied in \cite{bigot2012} (published as \cite{bigot2018}), where the authors, exploiting a duality argument, prove existence, uniqueness, and characterize the barycenter for compactly supported measures, and then study the convergence of the empirical barycenter to its population counterpart as the number of measures tends to infinity. In \cite{legouic} existence and consistence of the population barycenter is studied in a general geodesic space. Several applications of Wasserstein barycenters to analysis of data from demography and neuroimaging were given in the context of Fr\'echet regression in \cite{petersen}.

\subsubsection*{Metric extrapolation} The barycenter of two measures $\nu_0,\nu_1 \in \X$ with weights $\alpha>1$ and $\beta=1-\alpha<0$, that is
\begin{equation}
\label{eq:main}
	\min_{\mu\in \X} \alpha W_2^2(\mu,\nu_0) +\beta W_2^2(\mu,\nu_1),
\end{equation}
was studied in \cite{matthes-plazotta} and \cite{gallouet2024} in the context of metric extrapolation for the Backward Differentiation Formula of order 2 (BDF2). This connection requires some explanation. Given a metric space $(X,d)$ and a functional $\Phi:X\to \R\cup\{+\infty\}$, the gradient flow of $\Phi$ with respect to the metric $d$ is formally given by the solutions $u:[0,T]\to X$ of
\[    
	\frac{\d}{\d t} u(t) = -\nabla \Phi(u(t)),    
\]
or by its time-discrete equivalent $(u^k_\tau)\in X$
\begin{equation}
\label{eq:euler}
    \frac{u^k_\tau-u^{k-1}_\tau}{\tau} \approx -\nabla \Phi(u^k_\tau),
\end{equation}
for a small time-step $\tau>0$. A rigorous metric framework for the derivatives that appear in these equations was established in \cite{ambrosio-gigli-savare}. The minimizing movement scheme (also referred to as Implicit Euler method or Jordan-Kinderlehrer-Otto (JKO) stepping) provides, under suitable assumptions, a discrete approximation of the gradient flow equation, by solving a sequence of minimization problems defined by
\[    	
	u^k = \argmin_{w \in X}\left\{ \frac{1}{2\tau}d^2(u^{k-1}_\tau,w) +\Phi(w)\right\}.
\]
The BDF2 scheme is a second-order discretization scheme, well-known for ODEs in $\R^d$ since the 1950's (see, e.g., \cite{dahlquist1956}), which was recently proposed to approximate gradient flows in metric spaces \cite{matthes-plazotta}. In the Euclidean setting, the time-discrete  approximation \eqref{eq:euler} is substituted by
\[    
	\frac{3u^k_\tau - 4u^{k-1}_\tau +u^{k-2}_\tau}{2\tau} \approx -\nabla \Phi(u^k_\tau),
\]
which leads, in a JKO step, to the minimization of
\[    
	\Psi(\tau,u^{k-2}_\tau, u^{k-1}_\tau;w):= \frac{1}{\tau}d^2(u^{k-1}_\tau,w) -\frac{1}{4\tau}d^2(u^{k-2}_\tau,w) +\Phi(w).
\]
In \cite{matthes-plazotta} several examples of metric spaces $(X,d)$ are given, such that a sequence of piecewise-constant interpolations of the discrete solutions $(u^k_\tau)$ converges, locally uniformly in time, to a solution $u\in {\rm AC}^2([0,\infty),X)$ of the gradient flow of $\Phi$, in the sense of the following Evolutionary Variational Inequality (EVI)
\[
    \frac12d^2(u(t),w)-\frac12d^2(u(s),w) \leq \int_s^t \left[ \Phi(w) -\Phi(u(r)) -\frac{\lambda}{2}d^2(u(r),w)\right]\d r,
\]
for all $0\leq s<t$, where $\lambda \in \R$ is the $\lambda$-convexity modulus of $\Psi$ (\cite[Theorem 5.1]{matthes-plazotta}). Among examples given in \cite{matthes-plazotta}, we mention the following ones.
\begin{itemize}
	\item[(a)] In the case where $X$ is a Hilbert space, with the distance induced by its norm, the distance part of the JKO functional $\Psi$ is convex on straight segments and it is minimized by the weighted average of $u^{k-1}_\tau$ and $u^{k-2}_\tau$, as in a classical barycenter problem with positive coefficients.
	\item[(b)] In the case $(X,d)=(\X,W_2)$, minimizing the distance part of the JKO functional $\Psi$ is a Wasserstein barycenter problem like \eqref{eq:main}, with two assigned measures (corresponding to $u^{k-2}_\tau$ and $u^{k-1}_\tau$) and real coefficients. It should be remarked that the negative sign of one of the weights is crucial for the $\lambda$-convexity of $\Psi$ and thus for the existence and uniqueness of a minimizer. Indeed, we are going to exploit exactly \cite[Theorem 3.4]{matthes-plazotta} for the uniqueness part of our Theorem \ref{th:newmain}. 
\end{itemize}
Finally, we mention that two different convex formulations for \eqref{eq:main} were studied in \cite{gallouet2025}, together with an efficient numerical scheme to compute the minimizers.

\subsubsection*{Sticky particles} In \cite[Remarks 4.6 and 4.13]{gallouet2024}, an interesting connection was pointed out, between metric extrapolation and the one-dimensional pressureless Euler system
\begin{equation}
\label{eq:sticky}
\left\{
    \begin{array}{r}
        \partial_t \rho + \partial_x(\rho v) =0 \\
        \partial_t (\rho v) + \partial_x(\rho v^2) =0 
    \end{array}
  \right.  
\end{equation}
in $\R \times (0,+\infty)$, with initial density and velocity conditions $\rho_{|t=0}=\rho_0$, $v_{|t=0}=v_0$, for the evolution of a system of particles that share their trajectories after a collision (also called \emph{``sticky particle system"} (SPS)) (see also \cite{brenier}). Here, we exploit the characterization of the solutions to the (SPS) given in \cite{natile-savare}, in order to show a direct connection with the minimizer of a generalized Wasserstein barycenter functional. Let $(\rho_t,\rho_tv_t)$ be the solution of the (SPS) at time $t>0$, corresponding to a discrete initial datum
\[
	\rho_0 = \sum_{i=1}^N m_i \delta_{x_i},\qquad \rho_0v_0 = \sum_{i=1}^N m_i v_i \delta_{x_i},\qquad m_i>0,\ \sum_{i=1}^N m_i=1,\ x_i,v_i\in \R.
\]
Denoting $X_0$ (resp. $X_t$) the pseudo-inverse function of $\rho_0$ (resp. $\rho_t$), by \cite[Theorem 2.6 - II]{natile-savare} it holds
\begin{equation}
\label{eq:sps}
    X_t = P_\K(X_0+tV_0),
\end{equation}
where $\K$, as above, is the convex cone of nondecreasing functions in $L^2(0,1)$,   $P_\K$ is the orthogonal projection operator on $\K$, and $V_0$ is the piecewise-constant function such that $V_0(x)=v_i$ if $X_0(x)=x_i$. Let $\delta>0$ be the first collision time of the evolution, then, for all $0<s\leq \delta$
\begin{equation*}
    X_s = P_\K(X_0+sV_0) = X_0 +sV_0,
\end{equation*}
and owing to \eqref{eq:sps}, for all $t>s$
\begin{align*}
    X_t &= P_\K (X_0 +tV_0) \\
        &= P_\K \left(X_0 +t\frac{X_s-X_0}{s}\right) \\
        &= P_\K \left(\left(1-\frac{t}{s}\right)X_0 +\frac{t}{s}X_s\right).
\end{align*}
By \eqref{eq:charact}, we immediately see that the solution $\rho_t = (X_t)_\sharp \mathcal L^1\llcorner_{(0,1)}$ of the sticky particle system \eqref{eq:sticky} at time $t$ is also the minimizer of the generalized barycenter functional
\[
    \mathscr{E}(\rho)= \left(1-\frac{t}{s}\right)W_2^2(\rho_0,\rho) +\frac{t}{s}W_2^2(\rho_s,\rho),
\]
where the fixed measures are the initial density $\rho_0$ of the (SPS), with negative weight $1-t/s$ and the solution of the system at time $s$, that is, $\rho_s$, with positive weight $t/s$.

\subsection{Plan of the paper}
In Section \ref{sec:setting}, we recall the definitions and the main results we need concerning the Wasserstein distance, the weak topologies in spaces of probability measures, and signed measures. In Section \ref{sec:one} we study the one-dimensional case. After recalling Hoeffding's Fr\'echet's isometry, we state and prove the result concerning existence, uniqueness, and characterization of generalized barycenters for probability measures on the real line. In Subsection \ref{ssec:stable}, we provide a stability result and in Subsection \ref{ssec:gauss} an example where the generalized barycenter between one-dimensional Gaussian distributions is not a Gaussian. In Section \ref{sec:hilbert} we prove the existence part of Theorem \ref{th:newmain}, for measures on a separable Hilbert space, while in Section \ref{sec:unique} we prove uniqueness, in the case where $\lambda^+$ is concentrated on one single measure. In Subsection \ref{ssec:example}, we provide an example of non-uniqueness. Finally, in Section \ref{sec:gamma}, we prove Theorem \ref{th:gamma}.

\section{Setting and notation}
\label{sec:setting}
We collect here the minimal definitions needed to state our problem. For an in-depth treatise on optimal transport and Wasserstein distances we refer to the textbooks \cite{ambrosio-gigli-savare} and \cite{villani}.

\subsection{Wasserstein distance} We denote by  $H$ a real separable Hilbert space, endowed with the scalar product `\ $\cdot$\ ' and norm $|u|=\sqrt{u\cdot u}$. $\P(H)$ is the space of Borel probability measures $\mu$ on $H$ and $\X$ the subspace of measures $\mu$  with finite second moments: 
\begin{equation}
\label{def:m2}
	m_2^2(\mu):=\int_H |x|^2\d\mu(x) <+\infty.
\end{equation}
Given the metric spaces $(X_1,d_1)$ and $(X_2,d_2)$, for a general Borel map $f:X_1 \to X_2$ and a Borel measure $\mu$ on $X_1$, the push-forward measure of $\mu$ through $f$ is defined as
\[
	f_\sharp \mu (A) = \mu(f^{-1}(A)),\quad \text{for every Borel set }A\subseteq X_2.
\]
For $i=1,2$, let $\pi^i:H \times H \to H$ denote the projection operator on the $i^{th}$ variable, $\pi^i(x_1,x_2)=x_i$. Given two measures $\mu,\nu \in \X$, the set of admissible transport plans between $\mu$  and $\nu$ is
\[
	\Gamma(\mu,\nu)=\left\{ \gamma \in \P(H \times H) : \pi^1_\sharp \gamma = \mu,\ \pi^2_\sharp \gamma = \nu\right\}.
\]
The Wasserstein-Rubinstein-Kantorovich distance of order 2 between two measures $\mu,\nu \in \X$ is defined as
\[
	W_2(\mu,\nu) = \min\left\{ \left( \int_{H \times H} |x-y|^2 \d\gamma(x,y)\right)^{\frac12}  : \gamma \in \Gamma(\mu,\nu)\right\}.
\]
We remark that $W_2$ satisfies the axioms of a distance on $\X$: $W_2(\mu,\nu)=0$ if and only if $\mu=\nu$, and $W_2$ satisfies the triangle inequality (symmetry and nonnegativity follow from the other two properties).

We denote by $\Gamma_{\rm o}(\mu,\nu)$ the set of \emph{optimal transport plans} between $\mu,\nu\in \X$. That is, if $\gamma \in \Gamma_{\rm o}(\mu,\nu)$, then $\gamma \in \Gamma(\mu,\nu)$ and 
\[
	W_2^2(\mu,\nu)= \int_{H \times H}|x-y|^2\d\gamma(x,y).
\]
\begin{definition}[Narrow convergence]
We say that a sequence $(\mu_n)_{n\in \N} \subset \P(H)$ is narrowly convergent to $\mu \in \P(H)$ if
\[
	\lim_{n\to +\infty} \int_{H}f(x)\,\d\mu_n(x) = \int_{H}f(x)\,\d\mu(x)\qquad \forall f \in C_b(H).
\]
\end{definition}
We will need the following Gluing Lemma (\cite{ambrosio-gigli-savare}, Lemma 5.3.2)
\begin{lemma}
\label{lemma:glue}
Let $\nu_0,\nu_1,\nu_2\in \X$ and $\gamma^{0\, i}\in \Gamma(\nu_0,\nu_i)$ for $i=1,2$. Then there exists $\g \in \P\left( H \times H \times H\right)$ such that
\[
	(\pi^0,\pi^i)_\sharp \g =\gamma^{0\, i} \quad \text{for } i=1,2. 
\]
\end{lemma}
An example of such a $\g$ is given by the measure whose disintegration with respect to $\nu$ is 
\[
	\g = \int_{X} \left(\gamma_{\nu_0}^{0\,1}\times \gamma_{\nu_0}^{0\,2}\right) \d\nu_0,
\] 
where $	\gamma^{0\, i} = \int_{\R^d}\gamma^{0\, i}_{\nu_0}\d\nu_0$ are the disintegrations of the transport plans $\gamma^{0\, i}$ with respect to $\nu_0$.

\subsection{Weak topologies}
\label{ssec:weak}
In the proof of the existence of a minimizer for $\mathcal E$ we are going to use the framework of weak and strong-weak topologies of measures described in \cite{naldi-savare}. Let $C_2^w(H)$ be the Banach space of sequentially weakly continuous functions $\zeta: H \to \R$ satisfying
\[
	\lim_{|x| \to \infty}\frac{\zeta(x)}{1+|x|^2}=0,
\]
endowed with the norm
\[
	\| \zeta \|_{C_2^w(H)}= \sup_{x \in H} \frac{\zeta(x)}{1+|x|^2}.
\]
We define $\P_2^w(H)$ to be the topological space of measures in $\X$, endowed with the initial topology $\sigma(\X,C_2^w(H))$, i.e., the coarsest topology such that all functions 
\[
	F_\zeta:\X \to \R,\qquad F_\zeta(\mu):= \int_H \zeta(x)\d\mu(x),
\] 
with $\zeta \in C_2^w(H)$, are continuous.

In order to highlight the role of the different topologies, here we denote by $H_s$ the space $H$, in order to emphasize that we are using the strong topology, and $H_w$ the same space endowed with the weak topology. Let $Z= H_s \times H_w$ and 
\[
	\P_2(Z) := \left\{ \gamma \in \P(Z): \int_Z \left( |x|^2+|y|^2\right)\d \gamma(x,y)<+\infty \right\}.
\]
We consider the space $C^{sw}_2(Z)$ of test functions $\zeta :Z \to \R$ such that:
 \begin{itemize}
    \item $\zeta$ is sequentially continuous with respect to the strong-weak product topology on $Z$;
	\item $\forall \e >0\ \exists A_\epsilon \geq 0\ :\ |\zeta(x,y)| \leq A_\e \left( 1+|x|^2\right) +\e|y|^2.$
 \end{itemize}
We endow $C^{sw}_2(Z)$ with the norm
\[
	\|\zeta\|_{C^{sw}_2(Z)} := \sup_{(x,y)} \frac{|\zeta(x,y)|}{1+|x|^2 +|y|^2}.
\]
We endow $\P_2(Z)$ with the initial topology induced by the functions
\[
	\gamma \mapsto \int_Z \zeta(x,y)\, \d\gamma(x,y),\qquad \zeta \in C_2^{sw}(Z);
\]	
we call $\P_2^{sw}(Z)$ the topological space $(\P_2(Z),\sigma(\P_2(Z),C_2^{sw}(Z)))$. Proposition 3.4 in \cite{naldi-savare}, adapted to our simplified setting, reads as follows.
\begin{proposition}
\label{prop:NS}
A sequence $(\gamma_n)_{n\in \N} \subset \P_2^{sw}(Z)$ and a measure $\gamma \in \P_2^{sw}(Z)$ satisfy
\begin{itemize}
	\item[\rm{(i)}] $\gamma_n \weakto \gamma$ narrowly in $\P(Z)$;
	\item[\rm{(ii)}] $\displaystyle \lim_{n\to \infty} \int_Z |x|^2\, \d\gamma_n(x,y) \to \int_Z |x|^2\, \d \gamma(x,y)$;
	\item[\rm{(iii)}] $\displaystyle \sup_{n \in \N} \int_Z |y|^2\, \d\gamma_n(x,y) < +\infty$, 
\end{itemize}	
if and only if
\[
	\gamma_n \to \gamma\quad  \text{in}\quad \P_2^{sw}(Z).
\]
\end{proposition}
The following result \cite[Corollary 3.6]{naldi-savare} is an immediate consequence of Proposition \ref{prop:NS}.
\begin{corollary}
\label{cor:NS}
It holds
\begin{itemize}
	\item[\rm{(i)}] A sequence $(\mu_n)_{n\in \N}$ converges to $\mu$ in $\P_2^w(H)$ if and only if
	\[
		(\mu_n)_{n\in \N} \text{ converges narrowly in }\P(H_w)\quad \text{and}\quad \sup_{n \in \N}\int_H |x|^2\d\mu_n(x) < \infty.
	\]	
	\item[\rm{(ii)}] A set $\mathcal K \subset \X$ is relatively sequentially compact in $\P_2^w(H)$ if and only if
	\[
		\sup_{\mu \in \mathcal K}\int_H |x|^2\d\mu(x) < \infty.
	\]
\end{itemize}	
\end{corollary}

\subsection{Signed measures}
We refer to \cite{bogachev} for a comprehensive treatise of measure theory. Here, we will only use the following notions. Let $\mathcal B(X)$ be the Borel $\sigma$-algebra on the metric space $(X,d)$,  we denote by $\mathcal M^s(X)$ the space of countably additive signed measures on $\mathcal B(X)$. By Hahn's decomposition, there exist disjoint sets $\mathcal{A}^+$, $\mathcal{A}^- \in \mathcal B(X)$ such that  $\mathcal{A}^+ \cup \mathcal{A}^- =X$ and for all $B\in \mathcal{B}(X)$
\[	
	\lambda^+(B) :=\lambda(\mathcal{A}^+ \cap B) \geq 0,\qquad \lambda^-(B):=-\lambda(\mathcal{A}^- \cap B) \geq 0.
\]
Then $\lambda = \lambda^+ - \lambda^-$ is the (unique) Hahn-Jordan decomposition of $\lambda$ and the total variation measure is defined as $|\lambda|=\lambda^+ +\lambda^-$. For $f\in L^1(|\lambda|)$, we set
\[	
	\int_X f(x)\d\lambda(x) :=  	\int_X f(x)\d\lambda^+(x)  -	\int_X f(x)\d\lambda^-(x) .
\]
It is immediate to check that
\begin{equation}
\label{eq:signed_ineq}
	\left|\int_X f(x)\d\lambda(x)\right| \leq  	\int_X |f(x)|\d\lambda^+(x)  +	\int_X |f(x)|\d\lambda^-(x) = \int_X |f(x)|\,\d|\lambda|(x).
\end{equation}
We denote the variation of $\lambda$ by
\begin{equation}
\label{def:M}
	M(|\lambda|):= \lambda^+(X) +\lambda^-(X) = |\lambda|(X).
\end{equation}
Finally, for $X=\X$, the squared quadratic moment of $|\lambda|$ is 
\begin{equation}
\label{def:M2}
	M^2_2(|\lambda|):= \int_{\X}  m_2^2(\nu)\d|\lambda|(\nu).
\end{equation}


\section{The one-dimensional case}
\label{sec:one}

We restrict to the one-dimensional setting because in this case there is a standard isometry between $\P_2(\R)$ and the subset $\K$ of (essentially) nondecreasing functions in the Hilbert space of square-integrable functions $L^2(0,1)$. This relation allows us to exploit the properties of scalar products and projections in Hilbert spaces, which are crucial in our proof. We denote by $\| \cdot \|$ the standard norm in $L^2(0,1)$:
\[
	\|f\|=\left(\int_0^1 f^2(x)\d x \right)^{\frac 12}.
\]
Let $\mu\in \P_2(\R)$, the cumulative distribution function of $\mu$ is the function $F_\mu:\R \to [0,1]$
\[
		F_\mu(x):= \mu((-\infty,x])\qquad \forall\,x\in \R.
\]
Its pseudo-inverse (also called: `monotone rearrengement' or `quantile function') is the nondecreasing function $X_\mu: (0,1)\to \R$
\[
	X_\mu(w):=\inf \left\{ x: F_\mu(x)>w \right\}\qquad \forall\,w\in (0,1).
\]
Observe that $X_\mu(F_\mu(x))\geq x$ and that, if $m> F_\mu(x)$, then $X_\mu(m)> x$. Therefore, denoting by $\mathcal{L}^1$ the Lebesgue measure on $(0,1)$:
\[
    (X_{\mu})_\#\mathcal{L}^1 ((-\infty,x]) = \mathcal{L}^1 (X_\mu^{-1}((-\infty,x])) = \mathcal{L}^1((0, F_{\mu}(x)]) = F_\mu(x).
\]
In other words, $(X_{\mu})_\#\mathcal{L}^1 = \mu$.

Moreover, given two measures $\mu, \nu \in \P_2(\R)$, $\gamma = (X_\mu, X_\nu)_\# \mathcal{L}^1$ is the unique monotone transport in $\Gamma(\mu,\nu)$; it follows that $\gamma \in \Gamma_{\rm o}(\mu, \nu)$ (\cite[Theorem 2.9]{santambrogio}).
Therefore
\begin{equation}
\label{eq:hoeffding}
    W_2^2(\mu,\nu) = \int_{\R^2} |x-y|^2 \d (X_\mu, X_\nu)_\# \mathcal{L}^1 = \int_0^1 |X_\mu(x) - X_\nu(x)|^2 \d x = \|X_\mu - X_\nu \|^2.
\end{equation}
We have shown that the mapping $\phi: \P_2(\R)\to L^2(0,1)$, $\phi(\mu)=X_\mu$ satisfies the following properties (see also the \emph{Hoeffding-Fr\'echet} characterization of distributions with given marginals \cite[Sec. 3.1]{Rachev})):\begin{itemize}
	\item[(i)] for all $\mu,\nu \in \P_2(\R)$\quad  $\|\phi(\mu)-\phi(\nu)\|=W_2(\mu,\nu)$,
	\item[(ii)] for all $f \in \K$ the measure $\mu = f_\sharp \mathcal{L}^1 \in \P_2(\R)$ satisfies $\phi(\mu)=f$.
\end{itemize}	
Therefore, the metric spaces $\left( \P_2(\R),W_2\right)$ and $\left( \K,\| \cdot \|\right)$ are isometric. We recall that the orthogonal projection operator $P_\K:L^2(0,1) \to \K$ can be characterized by the following property (\cite{brezis}, Theorem 5.2): for all $f\in L^2(0,1)$, $P_\K(f)$ is the unique element in $\K$ such that
\begin{equation}
\label{eq:proj}
	\|P_\K(f)-f\|\leq \|g-f\|,\qquad \forall g\in \K.
\end{equation}  
Moreover, the orthogonal projection does not increase distance (\cite[Proposition 5.3]{brezis}):
\begin{equation}
\label{eq:nonexp}
	\|P_\K(f)-P_\K(g) \|\leq \|f-g\|,\qquad \forall f,g \in L^2(0,1).
\end{equation}
If $\lambda \in \mathcal M^s(\P_2(\R))$, $f\in \mathcal K$, and $\phi: \P_2(\R)\to L^2(0,1)$ is the Hoeffding-Fr\'echet isometry,  then $\eta:=\phi_\sharp \lambda \in \mathcal M^s(L^2(0,1))$ and
\begin{align*}
	\int_{L^2(0,1)}  \|f-g\|^2\d\eta(g) & = \int_{\P_2(\R)} \|f-\phi(\nu)\|^2\d\lambda(\nu) =\int_{\P_2(\R)} W_2^2(\phi^{-1}(f),\nu)^2\d\lambda(\nu).
\end{align*}  
In particular, 
\[
	M_2^2(|\lambda|)<+\infty\quad \text{if and only if}\quad M_2^2(|\eta |) = \int_{L^2(0,1)} \|g\|^2\d|\eta|(g) <+\infty.
\]

We can therefore give the following equivalent statement of Theorem \ref{th:intro}. 

\begin{theorem}
\label{th:1d}
Let $\eta$ be a signed measure on $L^2(0,1)$, such that supp$(\eta)\subset \K$,
\[
	\eta (L^2(0,1))=1\qquad \text{and}\qquad M_2(|\eta |)<\infty.
\]
Then, there exists a unique minimizer $f^* \in \K$ of the functional 
\[	
	E:L^2(0,1) \to \R,\qquad E(f) = \int_{L^2(0,1)}  \|f-g\|^2\d\eta(g)
\]
and it is characterized by
\[
	 f^* = P_\K\left( \int_{L^2(0,1)}g\,\d\eta(g)\right).
\]
\end{theorem}
The solution of Theorem \ref{th:intro} in the Wasserstein space can then be recovered by 
\[
	\mu^* =\phi^{-1}( f^*) =  f^*_\sharp \mathcal{L}^1.
\]	
\begin{remark}
\label{rem:general}
    Since the proof of Theorem \ref{th:1d} only relies on the Hilbert structure of the space $L^2(0,1)$ and on the closedness and convexity properties of $\K$, the same characterization of the solution holds in a more general setting. Precisely: let $(X,|\cdot|)$ be a real Hilbert space and let $K\subset X$ be a closed and convex subset. Let $\eta$ be a signed measure on $X$, such that 
 \begin{equation}
 \label{eq:hyp_gen}
 	\eta (X)=1\qquad \text{and}\qquad M_2(|\eta|)<\infty
\end{equation}	 
(we do not need that supp$(\eta)\subset K$). Then there exists a unique minimizer $x^* \in K$ of the functional 
\[
	G:X \to \R,\qquad G(x) = \int_X |x-y|^2\d\eta(y)
\]
and it is characterized by
\[
	 x^* = P_K\left( \int_X y\,\d\eta(y)\right).
\]
\end{remark}

\begin{remark}
    Theorem 3.1 in \cite{natile-savare} gives a useful characterization of the orthogonal projection $P_\K : L^2(0,1) \to \K$, which may be employed for explicit computations: given $f \in L^2(0,1)$, let $F(t) = \int_0^t f(s) \d s$ be its integral function; let $F^{**}$ be the lower semi-continuous convex envelope of $F$, i.e., the greatest lower semi-continuous convex function which is lower or equal to $F$. $F^{**}$ may also be defined by
    \[
    F^{**}(t) = \sup\{at + b : a,b\in \R,\ as + b \leq F(s) \text{ for a.e. } s \in (0,1)\}.
    \]
Being convex, in every point $F^{**}$ admits a left derivative $\frac{\\d^-}{\d t}F^{**}$ and a right derivative $\frac{\d^+}{\d t}F^{**}$. The projection may then be characterized by
    \[
    P_\K(f) (t) = \frac{\ \d^+}{\d t} F^{**}(t).
    \]
\end{remark}

\begin{proof}[Proof of Theorem \ref{th:1d}]
We adopt the more general setting introduced in Remark \ref{rem:general}: here $(X,|\cdot|)$ is a real Hilbert space and $K\subset X$ is a closed and convex subset. Let 
\[
	\bar x := \int_{X} y\,\d\eta(y)\qquad \text{and}\qquad \mathcal{C}(\eta):= \int_{X}|y|^2-|\bar x|^2 \d\eta(y).
\] 
These quantities are finite and only depend on $\eta$ since, by \eqref{eq:hyp_gen},
\begin{align*}
	|\bar x| &\leq \int_{X} |y|\,\d|\eta|(y) \leq \frac12\int_{X} \left(1+|y|^2\right)\,\d|\eta|(y) = \frac12\left(M(|\eta|+M_2^2(|\eta|))\right),\\
	|\mathcal{C}(\eta)| & \leq M_2^2(|\eta|) +|\bar x|^2.
\end{align*}
We compute 
\begin{align*}
	G(x)&= \int_X |x-y|^2\d\eta(y) \\
		&= \int_X |x|^2 -2x\cdot y +|y|^2\d\eta(y) \\
		&=  |x|^2 -2x\cdot \bar x +\int_X |y|^2\d\eta(y) \\
		&=  |x|^2 -2x\cdot \bar x +|\bar x|^2 +\int_X |y|^2\d\eta(y) -|\bar x|^2\\		
		&=  |x-\bar x|^2 +\mathcal{C}(\eta).
\end{align*}
Therefore, using the projection property \eqref{eq:proj} and exploiting the independence of $\mathcal{C}(\eta)$ from $x$, for all $z\in K$ 
\begin{align*}
	G\left( P_{K}(\bar x)\right)   & = |P_{K}(\bar x)-\bar x|^2 + \mathcal{C}(\eta)\\
		& \leq |z-\bar x|^2 + \mathcal{C}(\eta) \\
		&=G(z).
\end{align*}	
We conclude that $P_{K}(\bar x)$ is the unique minimum of $G$ in $K$.

\end{proof}

\subsection{The discrete case and stability}
\label{ssec:stable}
The barycenter of a finite sum of measures with positive and negative weights is a particular case of Theorem \ref{th:intro}. Precisely, let $n\geq 2$ be a given integer. For $i=1,\dots,n$, let $\nu_i\in \P_2(\R)$ be given and let $\lambda_i$  be real numbers such that $\sum_{i=1}^n\lambda_i =1$. We are choosing, that is, the signed measure:
\[
	\lambda = \sum_{i=1}^n\lambda_i \delta_{\nu_i},
\]
where $\delta_{\nu_i}$ is the Dirac measure concentrated in $\nu_i$. Then, by Theorem \ref{th:intro} there exists a unique solution $\bar \mu \in \P_2(\R)$ of
\[	
	\inf_{\mu \in \P_2(\R)} \sum_{i=1}^n \lambda_i W_2^2(\nu_i,\mu),
\]
and it is characterized by
\begin{equation}
\label{eq:1d_charact}
	X_{\bar \mu} = P_\K\left( \sum_{i=1}^n \lambda_i X_{\nu_i}\right),
\end{equation}
where $X_{\nu_i}\in \K$ is the pseudo-inverse function of (the distribution function of) $\nu_i$.

The explicit characterization of the discrete Wasserstein barycenter in one dimension entails a stability result, with respect to perturbations of the fixed measures:
\begin{lemma}
    Let $\nu_1, \ldots, \nu_n$ and $\Tilde{\nu}_1, \ldots, \Tilde{\nu}_n$ be probability measures in $\P_2(\R)$; let $\lambda_1, \ldots, \lambda_n$ be weights such that $\sum_{i=1}^n \lambda_i =1$ and consider the barycenters $\mu$, $\Tilde{\mu}$ with respect to the measures $\nu_1, \ldots, \nu_n$ and $\Tilde{\nu}_1, \ldots, \Tilde{\nu}_n$, with weights $\lambda_1, \ldots, \lambda_n$. Then
    \[
    W_2(\mu, \Tilde{\mu}) \leq \sum_{i=1}^n |\lambda_i|  W_2(\nu_i,\Tilde{\nu}_i).
    \]
\end{lemma}
\label{stability}
\begin{proof}
    For $\rho$ in $\P_2(\R)$, let $X_{\rho}$ be its monotone rearrangement.  Then, using \eqref{eq:hoeffding}, \eqref{eq:1d_charact}, and the non-expanding property of the projection \eqref{eq:nonexp}, we obtain
    \begin{align*}      
        W_2(\mu,\Tilde{\mu}) &= \left\| P_\K \left(\sum_{i=1}^{n} \lambda_i X_{\nu_i}\right) - P_\K \left(\sum_{i=1}^{n} \lambda_i X_{\Tilde{\nu}_i}\right)\right\| \\
		&\leq \left\| \sum_{i=1}^{n} \lambda_i X_{\nu_i} - \sum_{i=1}^{n} \lambda_i X_{\Tilde{\nu}_i}\right\| \\
         &\leq \sum_{i=1}^{n} |\lambda_i| \left\| X_{\nu_i} -  X_{\Tilde{\nu}_i}\right\| \\
         &=  \sum_{i=1}^n |\lambda_i|  W_2(\nu_i,\Tilde{\nu}_i).
    \end{align*}
\end{proof}


\subsection{An example in the case of two Gaussian measures}
\label{ssec:gauss}
While the (standard) barycenter of two Gaussian measures is always a Gaussian, it is possible to choose the parameters of the measures and a negative weight so that the generalized barycenter of two Gaussians is a Dirac delta.

\begin{remark}
\label{rem:invert}
If $\mu_1, \mu_2, \mu_3$ are three probability measures on $\R$ and $\mu_2$ is the barycenter of $\mu_1$ and $\mu_3$ with parameters $0<\lambda<1$ and $1- \lambda$, then $\mu_3$ is the generalized barycenter for $\mu_1$ and $\mu_2$ with parameters $(\lambda-1)/\lambda$, $1/\lambda$ (and, likewise, $\mu_1$ is a generalized barycenter for $\mu_2$, $\mu_3$ with parameters $1/(1-\lambda)$ and $-\lambda / (1-\lambda)$). This remark follows directly from the equation
\[
    X_{\mu_2} = (1 - \lambda) X_{\mu_1} + \lambda X_{\mu_3}
\]
since then 
\[
    X_{\mu_3} = \frac{1}{\lambda} X_{\mu_2} + \frac{\lambda - 1}{\lambda}X_{\mu_1} 
\]
and no projection on $\K$ is needed (despite the fact that $(\lambda-1)/\lambda<0$).
\end{remark}
Let $\mu_1 \sim \mathcal{N}(m_1,\sigma^2_1)$, $\mu_2 \sim \mathcal{N}(m_2, \sigma_2^2)$ be two Gaussian measures, and suppose that $\sigma_1>\sigma_2$: then it is possible to choose $\bar z \in \R$ such that $\mu_2$ is the barycenter between $\mu_1$ and the Dirac delta $\delta_{\bar{z}}$. Owing to Remark \ref{rem:invert}, $\delta_{\bar{z}}$ is then the generalized barycenter between the Gaussians $\mu_1$ and $\mu_2$. In order to choose $\bar z$, we proceed as follows. For $z\in \R$, the monotone rearrangement of the Wasserstein barycenter $\bar{\mu}$ between $\mu_1$ and an atomic measure $\delta_z$ satisfies
\[
    X_{\bar{\mu}} = (1-\lambda)X_{\mu_1} + \lambda z.
\]
Using the relation $(a X + b)^{-1}(x) = X^{-1}((x-b)/a)$, valid for all invertible functions $X:(0,1) \to \mathbb{R}$ and all $a,b$ in $\R$ with $a \neq 0$, we obtain
\[
    F_{\bar{\mu}} (x) = F_{\mu_1} ((x - \lambda z)/(1-\lambda)),
\]
so that the density of $\bar \mu$ is given by
\[
    F'_{\bar{\mu}} = \frac{1}{1-\lambda}f_{\mu_1}((x - \lambda z)/(1-\lambda)) = K \exp \left( - \frac{(x - ((1-\lambda) m_1 + \lambda z))^2}{2((1-\lambda)\sigma_1)^2}\right)
\]
where $K$ doesn't depend on $x$, so $\bar{\mu} \sim \mathcal{N}((1-\lambda) m_1 + \lambda z, ((1-\lambda)\sigma_1 )^2) $. Now we look for particular values $\bar{\lambda}$ and $\bar{z}$ for which $\bar{\mu} = \mu_2$: setting
\[
    \begin{cases}
        (1-\bar{\lambda}) m_1 + \bar{\lambda} \bar{z} = m_2, \\
        (1-\bar{\lambda})\sigma_1 = \sigma_2,
    \end{cases}
\]
yields 
\[	
	\bar{\lambda} = \frac{\sigma_1 - \sigma_2}{\sigma_1}, \quad \bar{z} = \frac{\sigma_1 m_2 - \sigma_2 m_1}{(\sigma_1-\sigma_2)}.
\]
With this choice of parameters, we have
\[
    X_{\mu_2} = (1-\bar\lambda)X_{\mu_1} + \bar\lambda \bar{z},
\]
and therefore, owing to Remark \ref{rem:invert}, the Dirac measure $\delta_{\bar z}$ is the generalized barycenter of the Gaussian measures $\mu_1$ and $\mu_2$, with weights $1 /\bar{\lambda}>0$ and $(\bar{\lambda}-1)/\bar\lambda<0$.

\section{Proof of Theorem \ref{th:newmain}-(i)}
\label{sec:hilbert}
In this section we employ the direct method of the calculus of variations to prove the existence part of Theorem \ref{th:newmain}. The crucial point is that the functional $\mathcal E$ has a negative term, due to $\lambda^-$, so that lower-semicontinuity of the integrand is not sufficient. In order to overcome this difficulty, in Lemma \ref{lemma:Rcont} we prove continuity of the coupling term in $W_2(\nu,\mu)$. The main tool is the framework of strong-weak topologies in $\P_2(H)$, introduced in \cite{naldi-savare} and recalled in Section \ref{ssec:weak}.

\begin{lemma}[Lower bound]
\label{lemma:lower}
Let $m_2,M$, and $M_2$ be defined as in \eqref{def:m2}, \eqref{def:M}, and \eqref{def:M2}. For all $\mu\in \X$
\begin{equation}
\label{eq:lower}
	\mathcal E(\mu) \geq \frac 12 m_2^2(\mu) - (1+2M(|\lambda|))M_2^2(|\lambda|).
\end{equation}
\end{lemma}

\begin{proof}
For $\mu,\nu\in \X$ and $\gamma \in \Gamma_{\rm o}(\mu,\nu$) 
\begin{align*}
	W_2^2(\mu,\nu) &= \int_{H \times H}|x-y|^2\d\gamma(x,y) \\
	&=\int_{H}|x|^2\d\mu(x) +  \int_{H}|y|^2\d\nu(y) -2 \int_{H \times H}x\cdot y\,\d\gamma(x,y).
\end{align*}
Therefore, for all $\delta>0$, by Young's inequality
\begin{align*}
	\left|W_2^2(\mu,\nu) -m_2^2(\mu)\right|&\leq m_2^2(\nu)+ 2 \left|\int_{H \times H}x\cdot y\,\d\gamma(x,y)\right| \\
	& \leq m_2^2(\nu)+ \int_{H \times H}\frac{|x|^2}{\delta} +\delta |y|^2\,\d\gamma(x,y)\\
	& =\frac{m_2^2(\mu)}{\delta} +(1+\delta) m_2^2(\nu).
\end{align*}
Using $\lambda(\X)=1$, \eqref{eq:signed_ineq}, and the last inequality, we obtain
\begin{align*}
	\mathcal{E}(\mu) & = \int_{\X} W_2^2(\mu,\nu)\d\lambda(\nu) \\
		& =m_2^2(\mu)+\int_{\X}\left( W_2^2(\mu,\nu)-m_2^2(\mu)\right)\d\lambda(\nu)\\
		& \geq m_2^2(\mu) - \int_{\X} \left( \frac{m_2^2(\mu)}{\delta} +(1+\delta) m_2^2(\nu) \right)\d|\lambda|(\nu)\\
				& =  \left(1-\frac{M(|\lambda|)}{\delta}\right)m_2^2(\mu) -(1+\delta)M_2^2(|\lambda|).
\end{align*}
Choosing $\delta=2M(|\lambda|)$ we obtain the thesis.
\end{proof}
Let 
\begin{equation*}
	\mathcal{R}(\nu,\mu):=W_2^2(\nu,\mu)-\int_{H}|x|^2\d\nu(x)-\int_{H}|y|^2\d\mu(y)=\min_{\gamma \in \Gamma(\nu,\mu)}\left\{ \int_{H \times H}-2x\cdot y\,\d\gamma(x,y)\right\}.
\end{equation*}

In the next Lemma, we prove the continuity of $\mathcal R$ with respect to the strong-weak convergence of measure. 
\begin{lemma}[Continuity of $\mathcal R$]
\label{lemma:Rcont}
Let $(\mu_n)_{n\in \N}\subset \X$ be a sequence converging to $\mu \in \X$ with respect to the (weak) topology of $\P_2^w(H)$ and let $(\nu_n)_{n\in \N}\subset \X$ be a sequence converging to $\nu \in \X$ with respect to the (strong) topology of $\P_2(H)$. 
Then, 
\begin{equation}
\label{eq:lsc_2}
	\lim_{n \to \infty}\mathcal R (\nu_n,\mu_n)=\mathcal R (\nu,\mu).
\end{equation}
\end{lemma}

\begin{proof}
For each couple $(\nu_n, \mu_n)$, there exists $\gamma_n \in \Gamma_{\rm o}(\mu_n,\nu_n)$ such that
\[
	\mathcal R(\nu_n, \mu_n)=\int_{H \times H}-2x\cdot y\,\d\gamma_n(x,y).
\]
We are going to show that $\gamma_n \to \gamma \in \Gamma_{\rm o}(\nu,\mu)$ narrowly in $\P(Z)$, in order to apply Proposition \ref{prop:NS} to $\mathcal R(\nu_n,\mu_n)$. First, we notice that by Corollary \ref{cor:NS}-(i),
\begin{equation}
\label{eq:lsc_1}
	\sup_{n\in \N}m_2^2(\mu_n)<+\infty.
\end{equation}
Since $(\pi_\sharp^1\gamma_n)=(\nu_n)$ is strongly converging in $\X$ (and thus tight in $\P(H_s)$) and $(\pi_\sharp^2\gamma_n)=(\mu_n)$ is tight in $\P(H_w)$ by \eqref{eq:lsc_1}, then $(\gamma_n)$ is tight in $\P(H_s \times H_w)$. Indeed, since $(\nu_n)$ is tight in $\P(H_s)$, there exists a lower-semicontinuous function $\psi:H\to [0,+\infty]$, with strongly compact sublevels, such that
\[
	\sup_{n\in\N}\int_H \psi(x)\d\nu_n(x)=S<+\infty.
\]  
Thus
\[
	\sup_{n\in\N}\int_{H\times H}\left(\psi(x) +|y|^2\right)\d\gamma_n(x,y) \leq S +\sup_{n\in \N}m_2^2(\mu_n) < +\infty.
\]
Since the sublevel sets of $\psi(x) +|y|^2$ are compact in $H_s \times H_w$, $(\gamma_n)$ is tight in $\P(H_s \times H_w)$. Let $\gamma$ be any narrow limit point of $(\gamma_n)$. Then for all $f\in C_b(H_s)$ and for all $g \in C_b(H_w)$
\[
	\int_{H\times H} f(x)\d\gamma(x,y) = \lim_{n \to \infty} \int_{H\times H} f(x)\d\gamma_n(x,y) =  \lim_{n \to \infty} \int_H f(x)\d\nu_n(x)  =\int_H f(x)\d\nu(x)
\]
and, by Corollary \ref{cor:NS}-(i),
\[
	\int_{H\times H} g(y)\d\gamma(x,y) = \lim_{n \to \infty} \int_{H\times H} g(y)\d\gamma_n(x,y) = \lim_{n \to \infty} \int_H g(y)\d\mu_n(y)  = \int_H g(y)\d\mu(y).
\]
Therefore, $\gamma\in \Gamma(\nu,\mu).$ Furthermore $\gamma \in \Gamma_{\rm o}(\mu,\nu)$ by ((\cite{naldi-savare}, Theorem 3.8). Since 
\[
	\lim_{n \to \infty} \int_Z |x|^2\, \d\gamma_n(x,y) = 	\lim_{n \to \infty} \int_{H} |x|^2\, \d\nu_n(x) = \int_{H} |x|^2\, \d\nu(x) =\int_Z |x|^2\, \d\gamma(x,y)
\]
and, by \eqref{eq:lsc_1},
\[
	\int_Z |y|^2\, \d\gamma_n(x,y) =\int_{H} |y|^2\, \d\mu_n(y) = m_2^2(\mu_n) <+\infty,
\]
by Proposition \ref{prop:NS} we conclude that $\gamma_n \to \gamma$ in $\P_2^{sw}(Z)$. The function 
\[
	\zeta:Z \to \R,\qquad \zeta(x,y)=-2x\cdot y
\]	 
is sequentially continuous with respect to the strong-weak product topology on $Z$ and by Young's inequality $\forall \e >0$
\[
	  |\zeta(x,y)|  \leq  A_\e \left(1+|x|^2\right)+\e|y|^2,
\]
with $A_\e=1/\e$. Therefore, exploiting the optimality of $\gamma_n$ and $\gamma$,
\begin{align*}
	\lim_{n\to \infty}\mathcal R(\nu_n,\mu_n)&= \lim_{n\to \infty} \left(W_2^2(\nu_n,\mu_n)-\int_H|x|^2\d\nu_n(x)-\int_H|y|^2\d\mu_n(y)\right)\\
	&= \lim_{n\to \infty} \int_{H \times H}-2x\cdot y\,\d\gamma_n(x,y) \\
	&= \int_{H \times H}-2x\cdot y\,\d\gamma(x,y) \\
	&= W_2^2(\nu,\mu)-\int_H|x|^2\d\nu(x)-\int_H|y|^2\d\mu(y)\\
	&=\mathcal R(\mu,\nu).
\end{align*}
\end{proof}

We have now all the elements to prove existence of a minimizer for $\mathscr{E}$.
\begin{proof}[Proof of Theorem \ref{th:newmain}-(i)]
By \eqref{eq:lower}
\[
	\underline{\mathcal{E}}:=\inf_{\mu \in \X} \mathcal E(\mu) \geq \inf_{\mu \in \X} \frac 12 m_2^2(\mu) - (1+2M)M_2^2(|\lambda|) >-\infty.
\]	
Let $(\mu_n)_{n\in \N}$ be a minimizing sequence, that is
\[
	\lim_{n \to \infty} \mathcal{E}(\mu_n)=\underline{\mathcal E}.
\]
By \eqref{eq:lower}
\[
	\sup_{n\in \N}m_2^2(\mu_n) =:S<+\infty
\]
and by Corollary \ref{cor:NS}-(ii) $(\mu_n)$ is relatively sequentially compact in $\P_2^w(H)$. Then, there exists a subsequence (which we do not relabel) and a limit point $\mu\in \X$ such that $\mu_n \to \mu$ with respect to the topology of $\P_2^w(H)$. Define the sequence of functions $f_n:\X \to \R$
\[
	f_n(\nu):=\mathcal R(\nu,\mu_n) = W_2^2(\nu,\mu_n)-m_2^2(\nu)-m_2^2(\mu_n)
\]
By \eqref{eq:lsc_2}, for all $\nu \in \X$
\[
	\lim_{n \to \infty} f_n(\nu) = \mathcal R(\nu,\mu) =:f(\nu).
\]	
For any $\gamma_n\in \Gamma_{\rm o}(\nu,\mu_n)$, by Young's inequality
\[
	|f_n(\nu)|=\left| \int_{H \times H}2x\cdot y\d\gamma_n(x,y)\right|\leq m_2^2(\nu)+m_2^2(\mu_n) \leq m_2^2(\nu) +S=:g(\nu).
\]
By hypothesis \eqref{eq:main_hyp}, $g\in L^1(\X,|\lambda|)$ and therefore, by Lebesgue's dominated convergence theorem
\begin{align}
	\lim_{n\to \infty} \int_{\X}f_n(\nu)\d\lambda(\nu) &=  \lim_{n\to \infty} \int_{\X}f_n(\nu)\d\lambda^+(\nu) - \lim_{n\to \infty} \int_{\X}f_n(\nu)\d\lambda^-(\nu)  \nonumber\\
		&=   \int_{\X}f(\nu)\d\lambda^+(\nu) - \int_{\X}f_n(\nu)\d\lambda^-(\nu)  \nonumber\\
		&= \int_{\X} f(\nu)\d\lambda(\nu).\label{eq:cont}
\end{align} 
The lower semicontinuity of $\mu \mapsto \int_H|y|^2\d\mu(y)=m_2^2(\mu)$ with respect to the narrow convergence of $\P(H)$ is a standard result. The lower semicontinuity with respect to the weak topology of $P_2^w(H)$ can be checked, e.g., by \cite[Theorem 5.1]{naldi-savare}, noticing that $m_2^2(\mu)=W_2^2(\mu,\delta_0)$. Finally, since
\begin{align*}
	\mathcal{E}(\mu) &= \int_{\X} W_2^2(\nu,\mu)\d\lambda(\nu)\\
		&= \int_{\X} \left( \int_H|x|^2\d\nu(x) + \int_H|y|^2\d\mu(y)+\mathcal{R}(\nu,\mu)\right)\d\lambda(\nu)\\
		&= m_2^2(\mu) +\int_{\X} \left(f(\nu) + m_2^2(\nu) \right)\d\lambda(\nu),
\end{align*}
owing to the lower semicontinuity of $m_2^2$ with respect to the topology of $P_2^w(H)$ and to the convergence in \eqref{eq:cont},
\begin{align*}
	\underline{\mathcal{E}}=\lim_{n\to \infty}\mathcal{E}(\mu_n) &\geq \liminf_{n\to \infty} m_2^2(\mu_n) +\lim_{n\to \infty}\int_{\X} \left(f_n(\nu) + m_2^2(\nu) \right)\d\lambda(\nu)\\
		&\geq m_2^2(\mu) +\int_{\X} \left(f(\nu) + m_2^2(\nu) \right)\d\lambda(\nu)\\
		&=\mathcal{E}(\mu),
\end{align*}
which shows that any weak limit point $\mu$ of $(\mu_n)$ is a minimizer of $\mathcal E$.
\end{proof}


\section{Uniqueness in Hilbert spaces}
\label{sec:unique}

In this section we specialize to the case where $\lambda^+$ is concentrated on a single measure $\nu_0\in\X$. In this case, uniqueness of the solution holds, even in contrast to the case of regular barycenters. 

In order to highlight the different role of the positive and negative weights, we use the following notation. Let $\alpha>0$ and $\nu_0\in \X$; let $\sigma$ be a positive measure on $\X$ such that 
\[
	\sigma(\nu_0)=0,\qquad \sigma(\X)=\alpha.
\]
We study 
\begin{equation}
\label{def:onepos}
	\mathscr{E}: \P_2(X) \to \R,\qquad \mathscr{E}(\mu) = (1+\alpha) W_2^2(\mu, \nu_0) - \int_{\X}  W_2^2(\mu, \nu)\d\sigma(\nu).
\end{equation}
With respect to the notation in the previous section, we are setting 
\[
	\lambda = \lambda^+-\lambda^-,\qquad \lambda^+=(1+\alpha)\delta_{\nu_0},\quad \lambda^-=\sigma.
\]	
\subsection{$\lambda$-convexity}
We will use the concept of $\lambda-$convexity along geodesics, as defined in \cite[Definition 9.2.4]{ambrosio-gigli-savare}; the argument, which we briefly expose below, is essentially the same as \cite[Theorem 4.1]{matthes-plazotta}: we adapt it to our functional.
Let $(M,d)$ be a complete metric space, and suppose we have an appropriate concept of connecting curve $\gamma(s) = \gamma_s: [0,1] \to M$ between two points $\gamma_0, \gamma_1$ in $M$, that we call \emph{generalized geodesic}. We say that a functional $F:M\to \R \cup \{+\infty\}$ is $\lambda-$convex along generalized geodesics (with $\lambda>0$) if for every $\gamma_0, \gamma_1$ in $M$, for every generalized geodesic $\gamma_s$ connecting $\gamma_0$ and $\gamma_1$, and for every $s$ in $[0,1]$, the inequality
\[
F(\gamma_s) \leq (1-s) F(\gamma_0) + s F(\gamma_1) - \lambda s (1-s) d^2(\gamma_0, \gamma_1)
\]
holds. If $F$ is $\lambda-$convex for some $\lambda>0$ and some class of generalized geodesics, if it is bounded from below, and it is lower-semicontinuous with respect to the metric topology of $M$, then $F$ admits a minimum in $M$. This is true because, taken a minimizing sequence $(\mu_i)_{i \in \N}$ in $M$, we have
\[
d^2(\mu_m, \mu_n) \leq 4/\lambda \left[ \frac{1}{2} (F(\mu_m) + F(\mu_n)) - F(\gamma_{1/2})\right]
\]
where $\gamma_s$ is a generalized geodesic connecting $\mu_m$ and $\mu_n$ and we have taken $s=1/2$.
Since the right-hand side term tends to $0$ as $m,n \to \infty$, the sequence is Cauchy; by completeness of $M$ and lower semi-continuity of $F$ we conclude that $\lim \mu_n$ is a minimum point for $F$.

Note that $\lambda-$convexity is stronger than strict convexity: if we have a minimum for $F$, it must also be unique.

If we specialize to the case $(M,d) = (\X, W_2)$, the right concept of generalized geodesic has been introduced in \cite[Definition 9.2.2]{ambrosio-gigli-savare}: we report here the definition and the basic properties.

\begin{definition}
        Let $\gamma_0$, $\gamma_1$ and $\nu_0$ be in $\X$. Choose optimal transports $\boldsymbol{\gamma}_{02}$ in $\Gamma_{\rm o}(\gamma_0,\nu_0)$ and $\boldsymbol{\gamma}_{12}$ in $\Gamma_{\rm o}(\gamma_1,\nu_0)$. By Lemma \ref{lemma:glue}, we can glue $\boldsymbol{\gamma}_{02}$ and $\boldsymbol{\gamma}_{12}$ along the common marginal $\nu_0$, obtaining a measure $\boldsymbol{\gamma}$ in $\P(H \times H \times H)$ such that $(\pi^{0,2})_{\#} \boldsymbol{\gamma} = \boldsymbol{\gamma}_{02}$ and $(\pi^{1,2})_{\#} \boldsymbol{\gamma} = \boldsymbol{\gamma}_{12}$. For $t\in (0,1)$, denote by $\pi^{0\to1}_t$ 
        the interpolating projection $(1-t) \pi^0 + t \pi^1:H^3 \to H$. A generalized geodesic $\gamma_t$ connecting $\gamma_0$ to $\gamma_1$, with basepoint $\nu_0$, is the curve $t \mapsto (\pi^{0\to1}_t)_{\#} \boldsymbol{\gamma} \in \X$ (see Fig. \ref{fig1}).
\end{definition}
\begin{center}
\begin{tikzpicture}
    \filldraw (0,0) circle (2pt);
    \node[yshift = .3cm] at (0,0) {$\nu_0$};
    \filldraw (2,1) circle (2pt);
    \node[yshift = .3cm] at (2,1) {$\gamma_0$};
    \filldraw (2,-1) circle (2pt);
    \node[yshift = -0.3cm] at (2,-1) {$\gamma_1$};
    \filldraw[gray] (2.3, .34) circle (2pt);
    \node[xshift = 0.4cm] at (2.3,.34) {$\gamma_t$};

    \draw [dotted, thick] (2,1) .. controls (2.5,0) .. (2,-1);
    \draw [thick] (0,0) -- (2,1);
    \draw[thick] (0,0) -- (2,-1);
    \draw[thick] (2,1) -- (2,-1);
\end{tikzpicture}
\captionof{figure}{Generalized geodesic with basepoint $\nu_0$ (dotted). The continuous lines denote regular geodesics induced by the optimal transports $(\pi^{0,2})_{\#} \boldsymbol{\gamma} = \boldsymbol{\gamma}_{02}$, $(\pi^{1,2})_{\#} \boldsymbol{\gamma} = \boldsymbol{\gamma}_{12}$, and an optimal transport between $\gamma_0$ and $\gamma_1$; the dotted line is the curve induced by the (not necessarily optimal) transport $(\pi^{0\to1}_t)_{\#} \boldsymbol{\gamma}$. 
}
\label{fig1}
\end{center}
Notation: Following \cite{ambrosio-gigli-savare}, we set 
$$
\pi^{i \to j}_t := (1-t) \pi^i + t \pi^j, \qquad \gamma^{i \to j}_t := (\pi^{i \to j}_t)_\# \boldsymbol{\gamma}.
$$
We will show that the functional $\mathscr{E}$ defined in \eqref{def:onepos} is $1-$convex; from this fact it will follow immediately that $\mathscr{E}$ has exactly one minimizer in $\P_2(X)$. Our proof follows the $\lambda$-convexity proof given in \cite[Theorem 3.4]{matthes-plazotta} in the case of metric extrapolation, i.e., with one positive and one negative coefficient. Nonetheless, it should be noted that extending that theorem to a higher number of measures should be done carefully, since there are precise constraints on the conditions that can be imposed on the marginals, in order for a common product measure to exist (see also Remark \ref{rem:basepoints} below).

We need the following lemma, where we adapt the notation to our case (see Fig. \ref{fig2}).

\begin{lemma}
    [Curve extension lemma; \cite{ambrosio-gigli-savare}, Proposition 7.3.1]
    Given $\gamma \in \P_2(H \times H)$, $\nu \in \X$, and $t \in [0,1]$, there exists $\boldsymbol{\nu}_t \in \P_2(H \times H \times H)$ such that $(\pi^{0,1})_{\#} \boldsymbol{\nu}_t = \gamma$ and $(\pi^{0\to1}_t,\pi^2)_{\#} \boldsymbol{\nu}_t \in \Gamma_{\rm o}((\pi^{0\to1}_t)_{\#}\gamma,\nu)$.
\end{lemma}

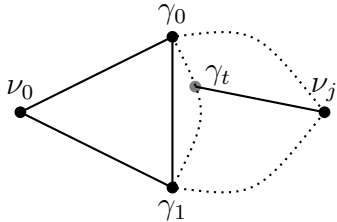
\begin{center}
\begin{tikzpicture}
    \filldraw (0,0) circle (2pt);
    \node[yshift = .3cm] at (0,0) {$\nu_0$};
    \filldraw (2,1) circle (2pt);
    \node[yshift = .3cm] at (2,1) {$\gamma_0$};
    \filldraw (2,-1) circle (2pt);
    \node[yshift = -0.3cm] at (2,-1) {$\gamma_1$};

    \filldraw[gray] (2.3, .34) circle (2pt);
    \node[xshift = 0.3cm, yshift = .15cm] at (2.3,.34) {$\gamma_t$};

    \filldraw (4,0) circle (2pt);
    \node[yshift = .3cm] at (4,0) {$\nu$};

    \draw [dotted,thick] (2,1) .. controls (2.5,0) .. (2,-1);
    \draw[dotted,thick] (4,0) .. controls (3.10,1.14) .. (2,1);
    \draw[dotted,thick] (4,0) .. controls (3.10,-1.14) .. (2,-1);
    \draw[thick] (4,0) -- (2.3,.34);
    \draw [thick] (0,0) -- (2,1);
    \draw[thick] (0,0) -- (2,-1);
    \draw[thick] (2,1) -- (2,-1);
\end{tikzpicture}
\captionof{figure}{Here's how we will use the lemma: with the same notation we used to define the generalized geodesic $\gamma_t$, $\gamma$ will be $(\pi^{0,1})_{\#} \boldsymbol{\gamma}$, $\nu_0$, is the fixed measure associated with the positive coefficient $(1+\alpha)$, $\nu$ will be integrated with respect to $\sigma=\lambda^-$ and $t$ will be a fixed time in $(0,1)$. This lemma yields a measure $\boldsymbol{\nu}_t$ in $\P(H \times H \times H)$ such that $(\pi^{0,1})_{\#} \boldsymbol{\nu}_t = (\pi^{0,1})_{\#}\g$ and $(\pi^{0\to1}_t, \pi^2)_{\#} \boldsymbol{\nu}_t \in \Gamma_{\rm o}(\gamma_t, \nu)$.}
\label{fig2}
\end{center}

\begin{theorem}
     The functional $\mathscr{E}$ defined in \eqref{def:onepos} is $1-$convex along generalized geodesics with basepoint $\nu_0 \in\X$.
\end{theorem}

\begin{proof}
    We need the equality, valid in Hilbert spaces:
    \begin{equation}
    \label{eq:hilbert}
    |(1-t)x_0 + tx_1 - x_2|^2 = (1-t)|x_0-x_2|^2 + t|x_1 - x_2|^2 - t(1-t)|x_0-x_1|^2.
    \end{equation}
    Let $\gamma_0, \gamma_1$ be in $\X$; let $\boldsymbol{\gamma}_{02}$ be in $\Gamma_{\rm 0}(\gamma_0, \nu_0)$ and $\boldsymbol{\gamma}_{12}$ be in $\Gamma_{\rm o}(\gamma_1, \nu_0)$. Let $\g$ be a measure in $\P_2(H \times H \times H)$ obtained by gluing the optimal transports along the common marginal $\nu_0$ and let $\gamma_t = (\pi^{0 \to 1}_t)_{\#} \boldsymbol{\gamma}$ be the associated generalized geodesic. Using \eqref{eq:hilbert} and the optimality of the marginals of $\boldsymbol{\gamma}$ we have
    \begin{equation*}
    \begin{split}
         W_2^2(\gamma_t, \nu_0) & \leq \int_{H \times H} |y_1 - y_2|^2 \d (\pi_t^{0\to 1},\pi^2)_\# \g (y_1,y_2)\\
        & = \int_{H \times H \times H} |(1-t)x_{0} + tx_1 - x_2|^2 \d\g(x_0,x_1,x_2)\\
        & = (1-t)W_2^2(\gamma_{0},\nu_0) + tW_2^2(\gamma_1,\nu_0) \\
        &\quad - (1-t)t \int_{H \times H \times H} |x_{0} - x_1|^2 \d\g(x_0,x_1,x_2).
    \end{split}
    \end{equation*}
    Now let $\nu \in \X$ and let $t \in [0,1]$. The curve extension lemma yields a measure $\boldsymbol{\nu}_t$ such that $(\pi^{0,1})_{\#} \boldsymbol{\nu}_t = (\pi^{0,1})_{\#}\g$ and $(\pi^{0\to 1}_t, \pi^2)_{\#} \boldsymbol{\nu}_t \in \Gamma_{\rm o}((\pi^{0\to1}_t)_{\#} \g, \nu)$. We have 
    \begin{equation*}
    \begin{split}
         W_2^2(\gamma_t, \nu) & = \int_{H \times H \times H} |(1-t)x_0 + tx_1 - x_2|^2 \d\boldsymbol{\nu}_t(x_0,x_1,x_2) \\
        & = (1-t) \int_{H \times H \times H} |x_0 - x_2|^2 \d\boldsymbol{\nu}_t \\
        &\quad + t \int_{H \times H \times H} |x_1 - x_2|^2 \d \boldsymbol{\nu}_t - t(1-t) \int_{H \times H \times H} |x_{0}-x_1|^2 \d\g \\
        & \geq (1-t) W_2^2(\gamma_{0}, \nu) + t W_2^2(\gamma_1, \nu) - t(1-t) \int_{H \times H \times H} |x_{0}-x_1|^2 \d\g.
    \end{split}
    \end{equation*}
    Using the two inequalities and the fact that $\sigma(\X)=\alpha$, we obtain
    \begin{equation*}
    \begin{split}
         \mathscr{E}(\gamma_t) &= (1+\alpha) W_2^2(\gamma_t, \nu_0) - \int_{\X}  W_2^2(\gamma_t, \nu)\d\sigma(\nu)\\
         &\leq (1-t) \mathscr{E}(\gamma_0) + t\mathscr{E}(\gamma_1) - t(1-t) \int_{H \times H \times H} |x_{0}-x_1|^2 \d\g \\
        &\leq (1-t) \mathscr{E}(\gamma_0) + t\mathscr{E}(\gamma_1) - t (1-t) W_2^2(\gamma_0, \gamma_1),
    \end{split}
    \end{equation*}
    which proves that $\mathscr{E}$ is $1-$convex along generalized geodesics with basepoint $\nu_0$.
\end{proof}
\begin{remark}
\label{rem:basepoints}
    Note that having only one positive weight is crucial: if we were to try to generalize the argument to $n\geq 2$ positive weights, we would have to construct $n$ different generalized geodesics $\gamma_t$ with basepoints equal to the different measures associated with the positive weights. 
\end{remark}

\subsection{Proof of Theorem \ref{th:newmain}-(ii)}
Now we repeat the informal argument presented above to show that $\mathscr{E}$ admits a unique minimizer.

\begin{corollary}\label{cauchy}
    Let $(\mu_i)_{i \in \mathbb{N}}$ be a minimizing sequence for $\mathscr{E}$. Then it is Cauchy with respect to $W_2$.
\end{corollary}
\begin{proof}
    We have already shown that $\inf \mathscr{E} > - \infty$. Let $l = \inf \mathscr{E}$ so that $\mathscr{E}(\mu_i) \to l$. Pick $\mu_m$, $\mu_n$ and let $\gamma_s$ be a generalized geodesic connecting them. For $s=\frac12$ we obtain
    $$
    \mathscr{E}(\gamma_{1/2}) \leq \frac{1}{2}(\mathscr{E}(\mu_m) + \mathscr{E}(\mu_n)) - \frac{1}{4}W_2^2(\mu_m,\mu_n)$$

    so that 
    $$
    \frac{1}{4}W_2^2(\mu_m,\mu_n) \leq \frac{1}{2}(\mathscr{E}(\mu_m) + \mathscr{E}(\mu_n)) - \mathscr{E}(\gamma_{1/2}) \leq \frac{1}{2}(\mathscr{E}(\mu_m) + \mathscr{E}(\mu_n)) - l
    $$
    and, since the right-hand term tends to $0$ as $m,n \to +\infty$, it follows that $(\mu_i)_{i \in \N}$ is Cauchy.
\end{proof}

Since $\mathscr{E}$ is continuous with respect to the metric $W_2$, we finally obtain the desired result. Here the positive part of $\lambda$ is concentrated on $\nu_0\in\X$, as in \eqref{def:onepos}.

\begin{proof}[Proof of Theorem \ref{th:newmain}-(ii)]
    Pick a minimizing sequence $(\mu_i)_{i \in \mathbb{N}}$: by Corollary \ref{cauchy}, it is a Cauchy sequence with respect to the metric $W_2$; since $(\P_2(X), W_2)$ is a complete metric space, there is a limit $\mu$. Then, by continuity of $\mathscr{E}$, 
    $$\inf \mathscr{E} = \lim_{i\to \infty} \mathscr{E}(\mu_i) = \mathscr{E}(\mu).$$ 
    
    Now suppose that there are two minimum points $\eta_1$ and $\eta_2$ for $\mathscr{E}$. Let $\gamma_s$ be a generalized geodesic with basepoint $\nu_0$. Then, for $s=1/2$,
    $$
    \frac{1}{4} W_2^2(\eta_1,\eta_2)\leq \frac{1}{2} (\mathscr{E}(\eta_1) + \mathscr{E}(\eta_2)) -\mathscr{E}(\gamma_{1/2}) \leq 0
    $$
    so that $W_2(\eta_1,\eta_2) = 0$: therefore $\eta_1 = \eta_2$.
\end{proof}


\subsection{Counterexample to uniqueness in the case of two positive weights}
\label{ssec:example}
We look at a particular example in $\P_2(\mathbb{R}^2)$: let 
\[
    \nu_0 = \delta_{(0,0)},\quad \nu_1 = \frac12 (\delta_{(-1,-1)} + \delta_{(1,1)}), \quad \nu_2 = \frac12 (\delta_{(1,-1)} + \delta_{(-1,1)}),
\]
and consider the functional
$$
\mathscr{E}(\mu) = - W_2^2(\mu, \nu_0) + W_2^2(\mu,\nu_1) + W_2^2(\mu,\nu_2).
$$
Owing to Theorem \ref{th:newmain}, $\mathscr E$ admits a minimizer in $\P_2(\mathbb{R}^2)$; we are going to prove that uniqueness does not hold.

The argument will employ the symmetry of the problem; for convenience, we define the matrix
\[
R = \begin{pmatrix}
    0 & 1 \\
    1 & 0
    \end{pmatrix}
\]
and the subsets
\[
    V = \left\{(x,y) \in \R^2 :  |y| > |x| \right\}  \quad \text{and}\quad
	U = \left\{(x,y) \in \R^2 :  |x| > |y| \right\}.
\]
\begin{center}
\begin{tikzpicture}
    \filldraw (0,0) circle (2pt);
    \draw (-1,-1) circle (1.3pt);
    \filldraw[dotted] (-1,1) circle (1.3pt);
    \filldraw[dotted] (1,-1) circle (1.3pt);
    \draw (1,1) circle (1.3pt);

    \draw[dotted,thick,->] (0,-1) -- (0,1);
    \draw[dotted,thick,->] (-1,0) -- (1,0);

	\node at (.3,.2) {$\nu_0$};
    \node at (.7,.8) {$\nu_1$};
    \node at (.7,-.8) {$\nu_2$};
    \node at (-.7,-.8) {$\nu_1$};
    \node at (-.7,.8) {$\nu_2$};
\end{tikzpicture}
\hspace{2cm}
\begin{tikzpicture}
    \filldraw (0,0) circle (1pt);
    \filldraw (-1,-1) circle (1pt);
    \filldraw (-1,1) circle (1pt);
    \filldraw (1,-1) circle (1pt);
    \filldraw (1,1) circle (1pt);

    \draw[thick] (0,0) -- (-1,-1);
    \draw[thick] (0,0) -- (-1,1);
    \draw[thick] (0,0) -- (1,-1);
    \draw[thick] (0,0) -- (1,1);

    \node at (0,.6) {$V$};
    \node at (.6,0) {$U$};
    \node at (0,-.6) {$V$};
    \node at (-.6,0) {$U$};
\end{tikzpicture}
\captionof{figure}{The fixed measured $\nu_0,\nu_1$, and $\nu_2$ (left); the subdivision of $\R^2$ using $U$ and $V$ (right).}
\label{fig3}
\end{center}
Denoting $\mathbf{x}=(x,y)\in \R^2,$ we consider the functions $S,T: \mathbb{R}^2 \to \mathbb{R}^2$,
\[
	S(\mathbf{x})= -(R\mathbf{x})\mathds{1}_V(\mathbf{x}) +\mathbf{x}\mathds{1}_{V^c}(\mathbf{x}),\quad \text{and}\quad T(\mathbf{x})= (R\mathbf{x})\mathds{1}_U(\mathbf{x}) +\mathbf{x}\mathds{1}_{U^c}(\mathbf{x}),
\]
where, for $A\subset \R^2$, $\mathds{1}_A$ is the characteristic function of the set $A$ and $A^c=\R^2\setminus A$. Let $\mu$ be a minimizer of $\mathscr{E}$; we assume that $\mu$ is unique and argue by contradiction. Owing to the symmetry of $\nu_0, \nu_1$, and $\nu_2$, (see Fig. \ref{fig3}), it can be  easily checked that
\[
	\mathscr E(S_\sharp \mu)  = \mathscr E(\mu) =	\mathscr E(T_\sharp \mu).
\]
By the uniqueness of $\mu$, we infer that $S_\sharp \mu = \mu$, and therefore supp$(\mu)$=supp$(S_\sharp \mu)\subset V^c$. Similarly $T_\sharp \mu = \mu$, and therefore supp$(\mu)$=supp$(T_\sharp \mu)\subset U^c$. In conclusion,
\[
	\text{supp}(\mu)\subset V^c \cap U^c = \left\{ (x,y)\in \R^2 : |x| = |y|\right\}.
\]
For all points $p$ in $\{|x| = |y| \}$ let us consider the function 
\[
	f(p) = \min \left\{|p - (-1,1)|^2, |p-(1,-1)|^2 \right\} + \min \left\{|p - (-1,-1)|^2, |p-(1,1)|^2 \right\} - |p|^2.
\]	 
Then 
\[
	\int_{\R^2} f(p) d\mu \leq \mathscr{E}(\mu).
\]
Assume that $p$ has the form $p=(x,x)$, then
\[
	|p - (-1,1)|^2 =  |p|^2 +2 = |p-(1,-1)|^2.
\]
The same computation, with respect to the support of $\nu_1$, holds if $p$ lies on $y=-x$. Therefore, for all measures $\mu$ concentrated on $\{|x| = |y| \}$, 
\[
	\mathscr{E}(\mu) \geq \int_{\R^2}\left( |p|^2+2-|p|^2\right)\d\mu(p)=2.
\]	
On the other hand, if we take $\eta = \frac12 \delta_{(0,1)} + \frac12 \delta_{(0,-1)}$, we have $\mathscr{E}(\eta) = 1$, which contradicts the assumption of $\mu$ being the minimizer of $\mathscr{E}$. We conclude that $\mathscr E$ does not admit a unique minimizer.


\section{Proof of Theorem \ref{th:gamma}}
\label{sec:gamma}
Let $(\lambda_k)_{k\in\N}$ be a sequence of signed measures on $\X$, such that 
\[	
	\lambda_k(\X)=1,\qquad M_2^2(|\lambda_k|)=\int_{\X}m_2^2(\nu)\, \d|\lambda_k|(\nu)<+\infty.
\]
As in the previous sections, the functional $\mathcal{E}_k:\X \to \R$ can be decomposed as
\begin{align*}
	\mathcal{E}_k(\mu)&= \int_{\X} W_2^2(\nu,\mu)\d\lambda_k(\nu) \\
		&=m_2^2(\mu) + \int_{\X} \left(\mathcal R(\nu,\mu) +m_2^2(\nu)\right)\d\lambda_k(\nu).
\end{align*}
We briefly recall the definitions equi-coercivity and of $\Gamma$-convergence, in the form we are going to employ. For a comprehensive treatise of the subject, we refer to  \cite{dalmaso}.
\begin{definition}
	We say that the sequence of functionals $\mathcal{E}_k:\X \to \R$ is \emph{equi-coercive} on $\P_2^w(H)$, if for all $t\in \R$ there exists a set $K_t \subset \X$, such that $K_t$ is relatively compact with respect to the topology of $\P_2^w(H)$ and $\{\mathcal E_k \leq t\} \subseteq K_t$.
\end{definition}
Owing to Corollary \ref{cor:NS}-(ii), in order to show the equi-coercivity of $(\mathcal E_k)$, it is equivalent to show that for all $t \in \R$ there exists a constant $C(t)>0$:
\[
	\text{if }\mu \in \{\mathcal E_k \leq t\}, \quad\text{then}\quad m_2^2(\mu)\leq C(t).
\] 
\begin{definition}	
	We say that the sequence of functionals $\mathcal{E}_k:\X \to \R$ $\Gamma$-converges to the functional $\mathcal{E}:\X \to \R$ in $\X$, with respect to the weak and strong topologies (or that $\mathcal E_k$ Mosco-converges to $\mathcal E$), if
	\begin{align*}
		&(i)\ \liminf_{k \to \infty} \mathcal E_k(\mu_k) \geq \mathcal E(\mu)\qquad \forall \mu_k \to \mu\quad \text{in }P_2^w(H);\\
		&(ii)\ \forall \mu\in \X \ \exists \mu_k \to \mu\ \text{in }(\P_2(H),W_2): \limsup_{k \to \infty}\mathcal E_k(\mu_k) \leq \mathcal E(\mu).
	\end{align*}	
\end{definition}

From now on we denote $X:=(\X,W_2)$ (i.e., the metric space $\X$, endowed with the 2-Wasserstein distance). Notice that \eqref{hyp:moment} implies (see, e.g., \cite[Lemma 6.1.5]{ambrosio-gigli-savare}) that $\nu \mapsto m_2^2(\nu)$ is uniformly integrable with respect to $(\lambda_k^+)$ and $(\lambda_k^-)$. Recalling that, by Lemma \ref{lemma:Rcont},  for all $\mu \in \X$ the function $\nu \mapsto \mathcal{R}(\nu,\mu)$ is continuous in $\P^w_2(H)$, and thus in $X$, we  immediately obtain that for all $\mu \in \X$,
\begin{equation}
\label{eq:contW}
	\lim_{k\to \infty} \int_{X} W_2^2(\mu,\nu)\d\lambda_k^{\pm}(\nu) = \int_{X} W_2^2(\mu,\nu)\d\lambda^{\pm}(\nu).
\end{equation}
In order to prove the $\Gamma$-liminf inequality, we need Skorohod's representation theorem (see, e.g., \cite[Theorem 8.5.4]{bogachev}), which we report here using our notation.
\begin{theorem}[Skorohod]
\label{th:skorohod}
Let $X$ be a complete, separable, metric space. Then, to every Borel probability measure $\mu$ on $X$, one can associate a Borel mapping $\xi_\mu: [0,1] \to X$ such that $\mu = (\xi_\mu)_\sharp \mathcal L$, where $\mathcal L$ is the Lebesgue measure, and 
\[
	\xi_{\mu_n}(t) \to 	\xi_{\mu}(t)\qquad \text{for almost all }t\in[0,1]
\]
whenever $\mu_n \to \mu$ narrowly in $\P(X)$.
\end{theorem}
We are going to apply Skorohod's representation to address the $\Gamma$-liminf of the term involving $\mathcal R(\nu,\mu)$.
\begin{lemma}
\label{lemma:gammaR}
Let $(\eta_k)_{k\in \N}$ and $\eta$ be Borel probability measures on $X=(\X,W_2)$, such that $\eta_k \to \eta$ narrowly and
\begin{equation}
\label{eq:m2conv}
	\lim_{k\to \infty} \int_{X}m_2^2(\nu)\,\d\eta_k(\nu) = \int_{X}m_2^2(\nu)\,\d\eta(\nu) < \infty.
\end{equation}
Then, for all $(\mu_k)_{k \in \N}$ and $\mu$ in $\X$ such that $\mu_k \to \mu$ with respect to the weak topology of $P_2^w(H)$,
\[
	\lim_{k \to \infty} \int_X \mathcal{R}(\nu,\mu_k)\d\eta_k(\nu) = \int_X \mathcal{R}(\nu,\mu)\d\eta(\nu).
\]
\end{lemma}
\begin{proof}
According to Skorohod's representation theorem, there exist Borel mappings $\xi_k,\xi : [0,1] \to X$ such that
\[
	(\xi_k)_\sharp \mathcal L = \eta_k,\quad 	\xi_\sharp \mathcal L = \eta,
\]
and $\xi_k(t) \to 	\xi(t)$ for almost all $t\in[0,1]$, i.e., 
\begin{equation}
\label{eq:xiconv}
	\lim_{k \to \infty} W_2(\xi_k(t),\xi(t)) =0,\qquad \text{for almost all }t\in[0,1].
\end{equation}
Therefore
\begin{equation}
\label{eq:skor}
	\int_{X} m_2^2(\nu) \d\eta_k(\nu) = \int_0^1 m_2^2(\xi_k(t)) \d t
\end{equation}
and
\[
	\int_{X} \mathcal R(\nu,\mu_k) \d\eta(\nu) = \int_0^1 \mathcal R(\xi_k(t),\mu_k) \d t.
\]
Define
\[
	f_k:[0,1]\to [0,+\infty),\qquad f_k(t) := m_2^2(\xi_k(t))=  \int_H |x|^2\d \xi_k(t)(x) \geq 0.
\]
Owing to \eqref{eq:xiconv}, we first remark that $f_k$ converges for almost all $t\in [0,1]$:
\[
	\lim_{k \to \infty} f_k(t) = \lim_{k \to \infty}  \int_H |x|^2\d \xi_k(t)(x) = \int_H |x|^2\d \xi(t)(x)= :f(t).
\]
Then, owing to \eqref{eq:skor} and \eqref{eq:m2conv} 
\[
	\lim_{k \to \infty} \int_0^1 f_k(t)\d t  = \lim_{k \to \infty} \int_{X} m_2^2(\nu) \d\eta_k(\nu) = \int_{X} m_2^2(\nu) \d\eta(\nu) = \int_0^1 f(t)\d t.
\]	
Since $f_k \to f$ almost everywhere and $\|f_k\|_{L^1} \to \|f\|_{L^1}$, we conclude that 
\begin{equation}
\label{eq:L1} 
	f_k \to f,\qquad \text{strongly in }L^1(0,1).
\end{equation}	
Since $\mu_k \to \mu$ (weakly) in $\P_2^w(H)$ and for almost all $t\in [0,1]$ $\xi_k(t) \to \xi(t)$ (strongly) in $\X$, by Lemma \ref{lemma:Rcont}
\[
	\lim_{k \to \infty} \mathcal R(\xi_k(t),\mu_k) = \mathcal R(\xi(t),\mu)\qquad \text{a.e. in }[0,1].
\]
Denoting $\gamma_k$ an optimal transport plan in $\Gamma_{\rm o}(\xi_k(t),\mu_k)$, we obtain
\begin{align*}
	|\mathcal R(\xi_k(t),\mu_k) | &= \left| 2\int_{H \times H} x\cdot y\, \d\gamma_k(x,y)\right | \\
		&\leq \left| \int_{H \times H} |x|^2+|y|^2 \d\gamma_k(x,y)\right | \\
		&=  m_2^2(\xi_k(t)) + m_2^2(\mu_k).
\end{align*}	
Since $m_2^2(\xi_k(t))=f_k(t)$ converges strongly in $L^1(0,1)$ and $m_2^2(\mu_k)$ is bounded by hypothesis (since $\mu_k \to \mu$ in $\P_2^w(H)$), we may conclude by dominated convergence that
\begin{align*}
    \lim_{k \to \infty}  \int_X \mathcal{R}(\nu,\mu_k) \d \eta_k(\nu) & = \lim_{k \to \infty} \int_0^1 \mathcal{R}(\xi_k(t),\mu_k) \d t \\
    & = \int_0^1 \mathcal{R}(\xi(t),\mu) \d t = \int_X \mathcal{R}(\nu,\mu) \d \eta(\nu).
\end{align*}
\end{proof}

\begin{lemma}[$\Gamma$-liminf]
\label{lemma:gammainf}
Let $(\lambda_k)$ and $\lambda$ be as in the hypotheses of Theorem \ref{th:gamma}. Then, for all $(\mu_k)$ and $\mu$ in $\X$ such that $\mu_k \to \mu$ with respect to the weak topology of $P_2^w(H)$,
\[
	\liminf_{k \to \infty} \mathcal{E}_k(\mu_k) \geq \mathcal{E}(\mu).
\]
\end{lemma}
\begin{proof}
We consider the Hahn-Jordan decomposition $\lambda_k=\lambda^+_k-\lambda^-_k$, $\lambda=\lambda^+-\lambda^-$, and notice that by \eqref{hyp:wconv}
\[
	\lim_{k \to \infty}\lambda^+_k(X)= \lambda^+(X),\qquad \lim_{k \to \infty}\lambda^-_k(X)= \lambda^-(X).
\]
Applying Lemma \ref{lemma:gammaR} to the sequences of probability measures
\[
	\eta^+_k=\frac{\lambda^+_k}{\lambda^+_k(X)}\qquad \text{and}\qquad \eta^-_k=\frac{\lambda^-_k}{\lambda^-_k(X)}
\]	 
we immediately obtain
\begin{align*}
	\lim_{k \to \infty}  &\int_X \mathcal{R}(\nu,\mu_k) \d \lambda_k(\nu) = 	\lim_{k \to \infty}  \int_X \mathcal{R}(\nu,\mu_k) \d \lambda^+_k(\nu) -	\lim_{k \to \infty}  \int_X \mathcal{R}(\nu,\mu_k) \d \lambda^-_k(\nu) \\
	&= \lim_{k \to \infty}  \lambda^+_k(X)\int_X \mathcal{R}(\nu,\mu_k) \d \eta^+_k(\nu) -	\lambda^-_k(X)\lim_{k \to \infty}  \int_X \mathcal{R}(\nu,\mu_k) \d \eta^-_k(\nu)\\
	&=  \int_X \mathcal{R}(\nu,\mu) \d \lambda^+(\nu) -	  \int_X \mathcal{R}(\nu,\mu) \d \lambda^-(\nu) = \int_X \mathcal{R}(\nu,\mu) \d \lambda(\nu).
\end{align*}
Recalling that
\begin{align*}
	\mathcal{E}_k(\mu_k)&= \int_{\X} W_2^2(\nu,\mu_k)\d\lambda_k(\nu) \\
		&=m_2^2(\mu_k) + \int_{\X} \left(\mathcal R(\nu,\mu_k) +m_2^2(\nu)\right)\d\lambda_k(\nu)
\end{align*}
    and that 
    \[
    \liminf_{k \to \infty} m_2^2(\mu_k) \geq m_2^2(\mu), \qquad \lim_{k\to \infty} \int_X m_2^2(\nu) \d \lambda_k(\nu) = \int_X m_2^2(\nu) \d \lambda (\nu),
    \]
    we obtain the thesis.
\
\end{proof}
\begin{lemma}[$\Gamma$-limsup]
\label{lemma:gammasup}
Let $(\lambda_k)$ and $\lambda$ be as in the hypotheses of Theorem \ref{th:gamma}. Then, for all  $\mu$ in $\X$ there exists a sequence $(\mu_k)$, converging to $\mu$ in $X$, such that
\[
	\limsup_{k \to \infty} \mathcal{E}_k(\mu_k) \leq \mathcal{E}(\mu).
\]
\end{lemma}
\begin{proof}
Let $\mu_k:=\mu$ for all $k\in \N$. Then, owing to \eqref{eq:contW},
\begin{align*}
	\limsup_{k \to \infty} \mathcal{E}_k(\mu_k) &= \limsup_{k \to \infty} \int_{X} W_2^2(\nu,\mu)\d\lambda_k(\nu)\\
		&= \limsup_{k \to \infty} \left(\int_{X} W_2^2(\nu,\mu)\d\lambda_k^+(\nu) - \int_{X} W_2^2(\nu,\mu)\d\lambda_k^-(\nu)\right)\\	
		&= \int_{X} W_2^2(\nu,\mu)\d\lambda^+(\nu) - \int_{X} W_2^2(\nu,\mu)\d\lambda^-(\nu)\\			
		&=  \int_{X} W_2^2(\nu,\mu)\d\lambda(\nu) = \mathcal{E}(\mu).
\end{align*}

\end{proof}
By definition of $\Gamma$ convergence, Lemma \ref{lemma:gammainf} and Lemma \ref{lemma:gammasup} ensure that $\mathcal{E}$ is the $\Gamma$-limit of $\mathcal{E}_k$ and complete the proof of part (i) of Theorem \ref{th:gamma}.

\begin{lemma}
The sequence of functionals $(\mathcal E_k)$ is equi-coercive with respect to the weak topology of $P_2^w(H)$.
\end{lemma}
\begin{proof}
By hypotheses \eqref{hyp:wconv} and \eqref{hyp:moment}, the variation and the quadratic moment of $\lambda_k$ are uniformly bounded, i.e., there exists $M_\infty>0$ such that
\[
	M(|\lambda_k|) +M_2^2(|\lambda_k|) < M_\infty\quad \forall k\in \N.
\]
By Lemma \ref{lemma:lower}, for all $\mu \in \X$
\[
	\mathcal E_k(\mu) \geq \frac 12 m_2^2(\mu) - (1+2M(|\lambda_k|))M_2^2(|\lambda_k|)\geq \frac 12 m_2^2(\mu) - (1+2M_\infty)M_\infty.
\]
Therefore, by Corollary \ref{cor:NS}-(ii), for all $t>0$ the set
\[
	\left\{ \mu \in \X: \mathcal{E}_k(\mu) \leq t\right\} \subset \left\{ \mu \in \X: m_2^2(\mu) \leq C(t,M_\infty)\right\}
\]
is relatively sequentially compact in $\P_2^w(H)$ (with $C(t,M_\infty)=2t+2M_\infty(1+M_\infty)$).
\end{proof}
If, in particular, $\mu_k$  is a minimizer of $\mathcal{E}_k$, then the sequence $(\mu_k)_{k\in \N}$ is precompact with respect to the weak topology of $\P_2^w(H)$, which is the statement in Theorem \ref{th:gamma}-(ii). Finally, by the fundamental theorem of $\Gamma$-convergence (\cite[Theorem 7.4]{dalmaso}), if $\mu_k$  is a minimizer of $\mathcal{E}_k$, then
\[
	\min_{\mu \in \X}\mathcal E(\mu) = \lim_{k \to \infty} \mathcal{E}_k(\mu_k),
\]
and if $(\mu_{k_j})_{j\in \N}$ is a subsequence converging to $\overline \mu$ in $\P_2^w(H)$, then (\cite[Corollary 7.20]{dalmaso}) 
\[
	\mathcal E(\overline \mu) = \min_{\mu \in \X}\mathcal E(\mu),
\]
which concludes the proof of Theorem \ref{th:gamma}-(iii).

\paragraph{Acknowledgements}

We would like to thank Alessandro Spelta and Stefano Gualandi for many stimulating and insightful discussions on the applicative and computational aspects of Wasserstein barycenters. We are also grateful to the anonymous reviewer for the incisive comments and suggestions, that, from our point of view, considerably improved our work.

\bibliography{bibliography}

\begin{thebibliography}{10}

\bibitem{agueh-carlier}
M.~Agueh and G.~Carlier.
\newblock Barycenters in the {W}asserstein space.
\newblock {\em SIAM J. Math. Anal.}, 43(2):904--924, 2011.

\bibitem{ambrosio-gigli-savare}
L.~Ambrosio, N.~Gigli, and G.~Savar\'e.
\newblock {\em Gradient flows in metric spaces and in the space of probability
  measures}.
\newblock Lectures in Mathematics ETH Z\"urich. Birkh\"auser Verlag, Basel,
  2005.

\bibitem{bigot2012}
J.~Bigot and T.~Klein.
\newblock Consistent estimation of a population barycenter in the {W}asserstein
  space.
\newblock arXiv preprint, 2012.

\bibitem{bigot2018}
J.~Bigot and T.~Klein.
\newblock Characterization of barycenters in the {W}asserstein space by
  averaging optimal transport maps.
\newblock {\em ESAIM: Probability and Statistics}, 22:35--57, 2018.

\bibitem{bogachev}
V.~I. Bogachev.
\newblock {\em Measure theory}.
\newblock Springer Berlin, Heidelberg, 2007.

\bibitem{brenier}
Y.~Brenier and E.~Grenier.
\newblock Sticky particles and scalar conservation laws.
\newblock {\em SIAM J. Numer. Anal.}, 35(6):2317--2328, 1998.

\bibitem{brezis}
H.~Brezis.
\newblock {\em Functional analysis, {S}obolev spaces and partial differential
  equations}.
\newblock Universitext. Springer, New York, 2011.

\bibitem{dahlquist1956}
G.~Dahlquist.
\newblock Convergence and stability in the numerical integration of ordinary
  differential equations.
\newblock {\em Math. Scand.}, 4:33--53, 1956.

\bibitem{dalmaso}
G.~Dal~Maso.
\newblock {\em An introduction to $\Gamma$-convergence}, volume~8.
\newblock Springer Science \& Business Media, 2012.

\bibitem{gallouet2024}
T.~O. Gallou\"et, A.~Natale, and G.~Todeschi.
\newblock From geodesic extrapolation to a variational {BDF}2 scheme for
  {W}asserstein gradient flows.
\newblock {\em Math. Comp.}, 93(350):2769--2810, 2024.

\bibitem{gallouet2025}
T.~O. Gallou{\"e}t, A.~Natale, and G.~Todeschi.
\newblock Metric extrapolation in the {W}asserstein space.
\newblock {\em Calculus of Variations and Partial Differential Equations},
  64(5):147, 2025.

\bibitem{legouic}
T.~Le~Gouic and J.-M. Loubes.
\newblock Existence and consistency of {W}asserstein barycenters.
\newblock {\em Probability Theory and Related Fields}, 168:901--917, 2017.

\bibitem{matthes-plazotta}
D.~Matthes and S.~Plazotta.
\newblock A variational formulation of the {BDF}2 method for metric gradient
  flows.
\newblock {\em ESAIM Math. Model. Numer. Anal.}, 53(1):145--172, 2019.

\bibitem{mc_cann}
R.~J. McCann.
\newblock A convexity principle for interacting gases.
\newblock {\em Adv. Math.}, 128(1):153--179, 1997.

\bibitem{mosco}
U.~Mosco.
\newblock Approximation of the solutions of some variational inequalities.
\newblock {\em Annali della Scuola Normale Superiore di Pisa-Scienze Fisiche e
  Matematiche}, 21(3):373--394, 1967.

\bibitem{naldi-savare}
E.~Naldi and G.~Savar\'{e}.
\newblock Weak topology and {O}pial property in {W}asserstein spaces, with
  applications to gradient flows and proximal point algorithms of geodesically
  convex functionals.
\newblock {\em Atti Accad. Naz. Lincei Rend. Lincei Mat. Appl.},
  32(4):725--750, 2021.

\bibitem{natile-savare}
L.~Natile and G.~Savar\'{e}.
\newblock A {W}asserstein approach to the one-dimensional sticky particle
  system.
\newblock {\em SIAM J. Math. Anal.}, 41(4):1340--1365, 2009.

\bibitem{petersen}
A.~Petersen and H.-G. M\"uller.
\newblock Fr\'echet regression for random objects with {E}uclidean predictors.
\newblock {\em Ann. Statist.}, 47(2):691--719, 2019.

\bibitem{Rachev}
S.~T. Rachev and L.~R\"{u}schendorf.
\newblock {\em Mass transportation problems. {V}ol. {I} : Theory}.
\newblock Probability and its Applications (New York). Springer-Verlag, New
  York, 1998.

\bibitem{santambrogio}
F.~Santambrogio.
\newblock {\em Optimal transport for applied mathematicians : Calculus of
  variations, PDEs, and modeling}, volume~87 of {\em Progress in Nonlinear
  Differential Equations and their Applications}.
\newblock Birkh\"{a}user/Springer, Cham, 2015.

\bibitem{villani_TOT}
C.~Villani.
\newblock {\em Topics in optimal transportation}, volume~58 of {\em Graduate
  Studies in Mathematics}.
\newblock American Mathematical Society, Providence, 2003.

\bibitem{villani}
C.~Villani.
\newblock {\em Optimal transport : Old and new}, volume 338 of {\em Grundlehren
  der mathematischen Wissenschaften [Fundamental Principles of Mathematical
  Sciences]}.
\newblock Springer-Verlag, Berlin, 2009.

\end{thebibliography}

\subsubsection*{Contacts}

Francesco Tornabene: Gran Sasso Science Institute, L'Aquila.
\smallskip

\noindent Marco Veneroni: Department of Mathematics, University of Pavia, Pavia. \textit{Email address:} \texttt{marco.veneroni@unipv.it} 
\smallskip

\noindent Giuseppe Savar\'e: Department of Decision Sciences, Bocconi University, Milano.

\end{document}